\documentclass[a4paper,12pt]{article}
\usepackage{amsmath,amssymb,amsthm,graphics,latexsym,amsfonts}
\usepackage{fancyhdr}
\usepackage{color}
\usepackage{graphics}
\usepackage{epstopdf}
\usepackage{amssymb}
\usepackage{cite}
\usepackage{hyperref}
\usepackage{tikz}
\usepackage{cleveref}

\hypersetup{
    colorlinks=true,
    linkcolor=cyan,
    filecolor=cyan,
    urlcolor=cyan,
    citecolor=cyan,
}

\title{\Large New results on the 1-isolation number of graphs without short cycles}
\author{ {Yirui Huang\,, Gang Zhang\,\footnotemark[1]\,\,, Xian'an Jin}\vspace{2mm}\\
	\small  School of Mathematical Sciences, Xiamen University,\\
	\small  Xiamen, Fujian 361005, P.R. China\\
}
\date{}
\newtheorem{theorem}{Theorem}[section]
\newtheorem{lemma}[theorem]{Lemma}

\newtheorem{problem}[theorem]{Problem}

\newcounter{claimcount}
\def\claimformat{\Alph{claimcount}}

\newenvironment{claim}{\refstepcounter{claimcount}\textbf{Claim \claimformat.}}{\par}

\usepackage{indentfirst}

\newcommand{\vertex}{\node[vertex]}
\tikzstyle{vertex}=[circle, draw, inner sep=0pt, minimum size=6pt]

\newenvironment{theorembis}[1]
{\renewcommand{\thetheorem}{\ref{#1}}%
	\addtocounter{theorem}{-1}%
	\begin{theorem}}
	{\end{theorem}}

\begin{document}
	
	\renewcommand{\thefootnote}{\fnsymbol{footnote}}
	\footnotetext[1]{Corresponding author.\\
		\hangindent=1.8em E-mail addresses: yrhuangis@163.com, gzh\_ang@163.com, xajin@xmu.edu.cn.}

	\maketitle {\small \noindent{\bfseries Abstract} Let $G$ be a graph. A subset $D \subseteq V(G)$ is called a 1-isolating set of $G$ if $\Delta(G-N[D]) \leq 1$, that is, $G-N[D]$ consists of isolated edges and isolated vertices only. The $1$-isolation number of $G$, denoted by $\iota_1(G)$, is the cardinality of a smallest $1$-isolating set of $G$. In this paper, we prove that if $G \notin \{P_3,C_3,C_7,C_{11}\}$ is a connected graph of order $n$ without $6$-cycles, or without induced 5- and 6-cycles, then $\iota_1(G) \leq \frac{n}{4}$. Both bounds are sharp.

\section {\large Introduction}

Every graph considered in this paper is finite, simple and undirected. We refer the readers to \cite{Bondy2008} for undefined notations and terminologies in graph theory, and refer to \cite{Zhang2021} for related notations and terminologies in this topic.

%\vspace{3mm}
Let $G$ be a graph and $k \geq 0$ be an integer. A subset $D \subseteq V(G)$ is a {\it $K_{1,k+1}$-isolating set} of a graph $G$ if $\Delta(G-N[D]) \leq k$, that is, $G-N[D]$ contains no $K_{1,k+1}$ as a subgraph. The {\it $K_{1,k+1}$-isolation number} of $G$, denoted by $\iota_k(G)$, is the cardinality of a smallest $K_{1,k+1}$-isolating set of $G$. Following Caro and Hansberg's definition in \cite{Caro2017}, the $K_{1,k+1}$-isolation is simply called the {\it $k$-isolation} in graphs.

In \cite{Caro2017}, Caro and Hansberg proved that

\begin{theorem}(Caro and Hansberg \cite{Caro2017}). \label{Theorem1.1} (i) For any graph $G$ of order $n$, $\iota_k(G) \leq \frac{n}{k+2}$. (ii) If $T$ is a tree of order $n$ that is different from $K_{1,k+1}$, then $\iota_k(T) \leq \frac{n}{k+3}$.
\end{theorem}

Moreover, the special cases of small $k$ receive more attention from scholars. Taking $k=0$, a subset $D \subseteq V(G)$ is an {\it isolating set} (instead of 0-isolating set, the same below) of $G$ if $G-N[D]$ is an {\it edgeless graph}, and the {\it isolation number} of $G$, denoted by $\iota(G)$, is the cardinality of a smallest isolating set of $G$.

Caro and Hansberg \cite{Caro2017} also proved the following pioneering theorem.

\begin{theorem}(Caro and Hansberg \cite{Caro2017}).
	If $G \notin \{K_2,C_5\}$ is a connected graph of order $n$, then $\iota(G) \leq \frac{n}{3}$.
\end{theorem}

Taking $k=1$, a subset $D \subseteq V(G)$ is a {\it 1-isolating set} of $G$ if $G-N[D]$ consists of isolated edges and isolated vertices only. The {\it 1-isolation number} of $G$, denoted by $\iota_1(G)$, is the cardinality of a smallest 1-isolating set of $G$. For maximal outerplanar graphs (mops), the following results on $\iota(G)$ and $\iota_1(G)$ were obtained.

\begin{theorem}
	If $G$ is a mop of order $n$, then the following holds:
	
	(i) (\cite{Caro2017}). If $n \geq 4$, then $\iota(G) \leq \frac{n}{4}$.
	
	(ii) (\cite{Borg&Kaemawichanurat2020}). If $n \geq 5$, then $\iota_1(G) \leq \frac{n}{5}$.
\end{theorem}

\begin{theorem}
	If $G$ is a mop of order $n$ with $n_2$ vertices of degree 2, then the following holds:
	
	(i)(\cite{Tokunaga2019}) If $n \geq 5$, then
	\begin{displaymath}
		\iota(G)\leq \left\{
		\begin{array}{ll}
			\frac{n+n_2} {5},& \text{if}\ n_2 \leq \frac {n} {4},\\
			\frac {n-n_2} {3},& \text{otherwise}.
		\end{array} \right.
	\end{displaymath}
	
	(ii)(\cite{Borg&Kaemawichanurat2020}) If $n\geq 5$, then
	\begin{displaymath}
		\iota_1(G)\leq \left\{
		\begin{array}{ll}
			\frac{n+n_2} {6},& \text{if}\ n_2 \leq \frac {n} {3},\\
			\frac {n-n_2} {3},& \text{otherwise}.
		\end{array} \right.
	\end{displaymath}
\end{theorem}

The {\it girth} of a graph $G$, denoted by $g(G)$, is the length of a shortest cycle of $G$. Zhang and Wu \cite{Zhang2021} investigated the 1-isolation number for general graphs, and they proved the following results.

\begin{theorem}(Zhang and Wu \cite{Zhang2021}).\label{Theorem1.5}
	If $G$ is a connected graph of order $n$, then the following holds: 
	
	(i)  If $G \notin \{P_3,C_3,C_6\}$, then $\iota_1(G) \leq \frac {2} 7 n$.
	
	(ii) If $G \notin \{P_3,C_7,C_{11}\}$ and $g(G) \geq 7$, then $\iota_1(G) \leq \frac{n}{4}$.
\end{theorem}

The readers are referred to \cite{Borg2020,Borg&Kaemawichanurat2023,Borg&Fenech2020,Favaron2021,Yan2022,Zhang2022} for more related problems and results on isolating sets in graphs. Note that every dominating set of a graph $G$ is a $k$-isolating set of $G$ for any $k \geq 0$. The research of isolation in graphs is a natural extension of the classical domination theory. For results on domination parameters in graphs with forbidden structures, we refer the readers to \cite{Bollobas1979,Brown2007,Chen2011,Cho2023,Dorbec2015,Harant2009,Henning2009}.

In this paper, we shall further study the 1-isolation number of graphs. We obtain the following two new results, each of them extending the results of Theorem \ref{Theorem1.1} (ii) when $k=1$ and Theorem \ref{Theorem1.5} (ii).

\begin{theorem}\label{Theorem1.6}
If $G \notin \{P_3,C_3,C_7,C_{11}\}$ is a connected graph of order $n$ without $6$-cycles, then $\iota_1(G) \leq \frac{n}{4}$.
\end{theorem}

\begin{theorem}\label{Theorem1.7}
    If $G \notin \{P_3,C_3,C_7,C_{11}\}$ is a connected graph of order $n$ without induced 5- and 6-cycles, then $\iota_1(G) \leq \frac{n}{4}$.
\end{theorem}

Both two bounds in Theorems \ref{Theorem1.6} and \ref{Theorem1.7} are best possible. We construct some extremal graphs as follows. Let $F$ be a connected graph of order $t \geq 1$ without $6$-cycles, or without induced 5- and 6-cycles. Suppose $V(F)=\{v_1,v_2,\cdots,v_t\}$. For each $i \in \{1,2,\cdots,t\}$, let $H_i$ be a graph isomorphic to a member of $\{P_3,C_3,C_7,C_{11}\}$. Let $G_t$ be the graph obtained from $F,H_1, H_2, \cdots,H_t$ (vertex-disjoint from each other) by joining $v_i$ to a vertex of $H_i$. Here, one can see Fig. \hyperlink{Fig1}{1} for an illustration of an example $G_4$.

Note that $G_t$ is a connected graph of order $n$ without $6$-cycles, or without induced 5- and 6-cycles, where $n = |V(G_t)|$. For each $i \in \{1,2,\cdots,t\}$, let $u_4^i$ be a vertex of $H_i \cong C_7$ which is distance 4 from $v_i$ in $G_t$, and $u_4^i$ and $(u_4^i)'$ be two vertices of $H_i \cong C_{11}$ which are distance 4 from $v_i$ in $G_t$. Set
\begin{displaymath}
	D_{i} = \left\{
	\begin{array}{ll}
		\{v_i\},& \text{if}\ H_i \in \{P_3,C_3\},\\
		\{v_i,u_4^i\},& \text{if}\  H_i \cong C_7,\\
		\{v_i,u_4^i,(u_4^i)'\},& \text{if}\  H_i \cong C_{11}.\\
	\end{array} \right.
\end{displaymath}
It is clear that $\bigcup_{1 \leq i \leq t}D_i$ is a 1-isolating set of $G_t$, and $|\bigcup_{1 \leq i \leq t}D_i| = \frac{1}{4}|V(G_t)| = \frac{n}{4}$.

On the other hand, let $S_i=\{v_i\} \cup V(H_i)$ for each $i \in \{1,2,\cdots,t\}$. It is easy to see that for any 1-isolating set $D$ of $G_t$, $|D \cap S_i| \geq \frac{1}{4}|S_i|$. Hence, we have $|D| \geq \frac{1}{4} \sum_{i=1}^{t} |S_i| = \frac{1}{4}|V(G_t)| = \frac{n}{4}$. Therefore, for any integer $t \geq 1$, $\iota_1(G_t)=\frac{n}{4}$.

\begin{figure}[h!]
    \begin{center}
    \begin{tikzpicture}[scale=.48]
    \tikzstyle{vertex}=[circle, draw, inner sep=0pt, minimum size=6pt]
    \tikzset{vertexStyle/.append style={rectangle}}
        \vertex (1) at (0,0) [scale=.75,fill=black] {};
        \vertex (2) at (0,1.5) [scale=.75,fill=lightgray] {};
        \vertex (3) at (-1.4,2.5) [scale=.75] {};
        \vertex (4) at (-1.4,3.5) [scale=.75,fill=lightgray] {};
        \vertex (5) at (-0.6,4.5) [scale=.75,fill=black] {};
        \vertex (6) at (0.6,4.5) [scale=.75,fill=lightgray] {};
        \vertex (7) at (1.4,3.5) [scale=.75] {};
        \vertex (8) at (1.4,2.5) [scale=.75] {};

        \path
        (1) edge (2)
        (2) edge (3)
        (3) edge (4)
        (4) edge (5)
        (5) edge (6)
        (6) edge (7)
        (7) edge (8)
        (2) edge (8)
        ;

        %\draw [line width=0.8pt] (0,-0.7)..controls (0.7,-0.5) and (0.7,2) ..(0,2.2) ..controls (-0.7,2) and (-0.7,-0.5) .. (0,-0.7)[dotted] ;

        %\draw [line width=0.8pt] (-0.7,4.9) ..controls (-2.7,3)  and (-1.3,2.4) ..(-0.4,4.1) [dotted] ;

        %\draw [line width=0.8pt] (-0.5,4.9) ..controls (1.6,5.2)  and (1.6,3.9) ..(-0.3,4.1) [dotted] ;

        \vertex (9) at (-4.7,0) [scale=.75,fill=black] {};
        \vertex (10) at (-4.7,1.5) [scale=.75,fill=lightgray] {};
        \vertex (11) at (-5.9,3) [scale=.75] {};
        \vertex (12) at (-3.5,3) [scale=.75] {};

        %\draw [line width=0.8pt] (-4.7,-0.7)..controls (-4,-0.5) and (-4,2) ..(-4.7,2.2) ..controls (-5.4,2) and (-5.4,-0.5) .. (-4.7,-0.7)[dotted] ;

        \path
        (9) edge (10)
        (10) edge (11)
        (12) edge (10)
        ;
        
        \path
        (11) edge (12) [dashed]
        ;

        \vertex (13) at (3.6,0) [scale=.75,fill=black] {};
        \vertex (14) at (3.6,1.5) [scale=.75,fill=lightgray] {};
        \vertex (15) at (3.6,3) [scale=.75] {};
        \vertex (16) at (3.6,4.5) [scale=.75] {};

        \path
        (13) edge (14)
        (14) edge (15)
        (15) edge (16)
        ;

        %\draw [line width=0.8pt] (3.6,-0.7)..controls (4.3,-0.5) and (4.3,2) ..(3.6,2.2) ..controls (2.9,2) and (2.9,-0.5) .. (3.6,-0.7)[dotted] ;

        \vertex (17) at (9,0) [scale=.75,fill=black] {};
        \vertex (18) at (9,1.5) [scale=.75,fill=lightgray] {};
        \vertex (19) at (7.7,1.7) [scale=.75] {};
        \vertex (20) at (6.5,2.4) [scale=.75,fill=lightgray] {};
        \vertex (21) at (6.5,3.6) [scale=.75,fill=black] {};
        \vertex (22) at (7.2,4.4) [scale=.75,fill=lightgray] {};
        \vertex (23) at (8.4,4.7) [scale=.75] {};
        \vertex (24) at (9.6,4.7) [scale=.75] {};
        \vertex (25) at (10.8,4.4) [scale=.75,fill=lightgray] {};
        \vertex (26) at (11.5,3.6) [scale=.75,fill=black] {};
        \vertex (27) at (11.5,2.4) [scale=.75,fill=lightgray] {};
        \vertex (28) at (10.3,1.7) [scale=.75] {};

        %\draw [line width=0.8pt] (9,-0.7)..controls (9.7,-0.5) and (9.7,2) ..(9,2.2) ..controls (8.3,2) and (8.3,-0.5) .. (9,-0.7)[dotted] ;

        %\draw [line width=0.8pt] (6.2,2)..controls (5.4,4.2) and (6.7,4.3) ..(7.5,4.8)..controls (7.2,4.1) and (6.5,2.1) .. (6.2,2)[dotted] ;

        \path
        (17) edge (18)
        (18) edge (19)
        (19) edge (20)
        (20) edge (21)
        (21) edge (22)
        (22) edge (23)
        (23) edge (24)
        (24) edge (25)
        (25) edge (26)
        (26) edge (27)
        (27) edge (28)
        (28) edge (18)
        ;

        \path
        (1) edge (9)
        (1) edge (13)
        (13) edge (17)
    
        ;

        \draw (-4.5,-0.2) parabola bend (2.3,-2) (8.9,-0.1) [dashed];

    \end{tikzpicture}
\end{center}
\par {\footnotesize \centerline{{\bf Fig. 1.} ~An extremal graph $G_4$ with $\iota_1(G_4) = \frac{n}{4}$. \hypertarget{Fig1}}}
\end{figure}

\noindent{\it Remark}. (i) The bounds in Theorems \ref{Theorem1.6} and \ref{Theorem1.7} can both be improved to $\lfloor \frac{n}{4} \rfloor$. (ii) For any positive integer $n$, there exists a graph $G_t'$ of order $n$ such that $\iota_1(G_t') = \lfloor \frac{n}{4} \rfloor$. In fact, $G_t'$ can be obtained from $G_t$ by adding at most three leaves to some vertices of $\bigcup_{1 \leq i \leq t}D_i$.

%(ii) There exist a number of extremal graphs attaining the bound $\iota_1(G) = \lfloor \frac{n}{4} \rfloor$. Let $G_t'$ be the graph obtained from $G_t$ by adding at most three leaves to some vertices of $\bigcup_{1 \leq i \leq t}D_i$. Then, $\iota_1(G_t') = \lfloor \frac{n}{4} \rfloor$.
    
%Let $G_t'$ be a graph constructed by adding some vertices to $G_t$ so that these vertices are adjacent to some vertices in the 1-isolating set $D$ of $G_t$ and have degree 1 in $G$.

\section {\large Preliminaries}

In this section, we state some definitions and lemmas that will be used frequently in the proofs of our main results. 

\begin{lemma}(Caro and Hansberg \cite{Caro2017}).\label{Lemma2.1} (i) If $T$ is a tree different from $P_3$, then $\iota_1(T)\leq \frac{n}{4}$. (ii) If $G \notin \{C_3,C_6,C_7,C_{11}\}$ is a cycle, then $\iota_1(G)\leq \frac{n}{4}$.
\end{lemma}

Let $X,Y \subseteq V(G)$ be two disjoint vertex subsets of a graph $G$, and let $E(X,Y)$ be the set of edges of $G$ with one end in $X$ and the other end in $Y$. 

\begin{lemma}(Zhang and Wu \cite{Zhang2021}).\label{Lemma2.2}
    Let $G=(V,E)$ be a graph. For any $S \subseteq V(G)$, if $G[S]$ has a 1-isolating set $D$ such that $E(S \setminus N[D],V \setminus S)= \emptyset$, then $$\iota_1(G) \leq |D|+\iota_1(G-S).$$
\end{lemma}

\begin{lemma}(Zhang and Wu \cite{Zhang2021}).\label{Lemma2.3}
    Let $\mathcal{H}(G)$ be the set of connected components of $G$. Then $$\iota_1(G)= \sum_{H\in \mathcal{H}(G)}\iota_1(H).$$
\end{lemma}

For convenience, we define the set of graphs $\mathcal{S}=\{P_3,C_3,C_7,C_{11}\}$. A graph is called an {\it $\mathcal{S}$-graph} if it is isomorphic to a member of $\mathcal{S}$. 

\section {\large Proof of Theorem \ref{Theorem1.6}}

%\begin{theorembis}{Theorem1.6}
%    If $G \notin \{P_3,C_3,C_7,C_{11}\}$ is a connected graph of order $n$ without $6$-cycles, then $\iota_1(G) \leq \frac{n}{4}$.
%\end{theorembis}

In this section, we prove Theorem \ref{Theorem1.6}. Let $G=(V,E)$ be a connected graph of order $n$ without 6-cycles. Suppose that $G$ is not an $\mathcal{S}$-graph, that is, $G \notin \{P_3,C_3,C_7,C_{11}\}$. The proof is by induction on $n$. If $G \notin \{P_3,C_3\}$ has order $n \leq 3$, then $G \in \{K_1,K_2\}$, and then $\iota_1(G) = 0 < \frac{n}{4}$. Hence, we may assume that $n \geq 4$.

If $\Delta(G) \leq 2$, then $G \in \{P_n,C_n\}$. By Lemma \ref{Lemma2.1}, $\iota_1(G)\leq \frac{n}{4}$. So, let $\Delta(G) \geq 3$. Fix a vertex $v \in V(G)$ with $d(v)=\Delta(G)$. If $\Delta(G)=n-1$, then $\iota_1(G)\leq |\{v\}| = 1 \leq \frac{n}{4}$. Hence, we may assume that $3 \leq \Delta(G) \leq n-2$.

We further consider $G'=G-N[v]$. Let $|V(G')|=n'$. Since $\Delta(G)\leq n-2$, $n' \geq 1$. Let $\mathcal{H} \neq \emptyset$ be the set of components of $G'$,  $\mathcal{H}_b$ be the set of components of $G'$ isomorphic to an $\mathcal{S}$-graph, and $\mathcal{H}_g=\mathcal{H} \setminus \mathcal{H}_b$. By the induction hypothesis, $\iota_1(H) \leq \frac{1}{4}|V(H)|$ for any component $H \in \mathcal{H}_g$.

\vspace{3mm}

\begin{claim} \label{ClaimA}
    $\mathcal{H}_b \neq \emptyset$. 
\end{claim}

\begin{proof}
Suppose to the contrary that $\mathcal{H}_b = \emptyset$. Then, $\mathcal{H}=\mathcal{H}_g \neq \emptyset$. It is easy to see that the set $\{v\}$ is a 1-isolating set of $G[N[v]]$. By Lemmas \ref{Lemma2.2} and \ref{Lemma2.3}, and by the induction hypothesis, we have $\iota_1(G) \leq |\{v\}| + \iota_1(G') =1+\sum_{H \in \mathcal{H}} \iota_1(H) \leq 1+\frac{1}{4}(n-\Delta(G)-1) \leq \frac{n}{4}$. The result follows.
\end{proof}

For each $H \in \mathcal{H}$, we denote that $N(H)=N(V(H))$ simply. For any $x \in N(v)$, let $\mathcal{H}_b^x$ be the set of components $H$ of $\mathcal{H}_b$ with $N(H) = \{x\}$, and $\mathcal{H}_g^x$ be the set of components $H$ of $\mathcal{H}_g$ with $N(H) = \{x\}$.

\vspace{3mm}

{\bf Case 1.} For some $x \in N(v)$, $\mathcal{H}_b^x \neq \emptyset$.

\vspace{3mm}

Let $k_3$ be the number of components isomorphic to $P_3$ or $C_3$ in $\mathcal{H}_b^x$, and $k_i$ be the number of components isomorphic to $C_i$ in $\mathcal{H}_b^x$, where $i \in \{7,11\}$. By the present assumption, $k_3+k_7+k_{11} \geq 1$. Let $X = \{x\}\cup \bigcup_{H \in \mathcal{H}_b^x}V(H)$. Then, $G-X= G_v \cup \bigcup_{H \in \mathcal{H}_g^x}H$, where $G_v$ is the component of $G-X$ containing $v$.
    
For each $H \in \mathcal{H}_b^x$, let $xy \in E(G)$ for some $y \in V(H)$, and if $H \in \{C_7, C_{11}\}$, then let $y_3$, $y_3'$ be the two vertices of $H$ which are distance 3 from $y$ on the cycle. Take
\begin{displaymath}
    D_{H} = \left\{
        \begin{array}{ll}
        \{x\},& \text{if}\ H \in \{P_3, C_3\},\\
        \{x, y_3\},& \text{if}\  H \cong C_7,\\
        \{x, y_3,y_3'\},& \text{if}\  H \cong C_{11}.\\
        \end{array} \right.
\end{displaymath}
As shown in Fig. \hyperlink{Fig2}{2}, $D_X=\bigcup_{H \in \mathcal{H}_b^x} D_{H}$ is a 1-isolating set of $G[X]$. Clearly, $$|D_X|= | \bigcup_{H \in \mathcal{H}_b^x} D_{H} | = \sum_{H \in \mathcal{H}_b^x} | D_{H} \setminus \{x\} |+|\{x\}|= 1+k_7+2k_{11}.$$ 

\begin{figure}[h!]
    \begin{center}
    \begin{tikzpicture}[scale=.48]
    \tikzstyle{vertex}=[circle, draw, inner sep=0pt, minimum size=6pt]
    \tikzset{vertexStyle/.append style={rectangle}}
            \vertex (1) at (0,0) [scale=.75,fill=black] {};
            \node ($x$) at (0.5,-0.5) {$x$};
            \vertex (2) at (0,-2) [scale=.75,fill=lightgray] {};
            %\node ($y$) at (0.9,-1.9) {$y$};
            \vertex (3) at (-1,-2.8) [scale=.75] {};
            \vertex (4) at (-1,-3.6) [scale=.75,fill=lightgray] {};
            \vertex (5) at (-0.5,-4.4) [scale=.75,fill=black] {};
            %\node ($y_{3}$) at (-1.1,-4.9) {$y_{3}$};
            \vertex (6) at (0.5,-4.4) [scale=.75,fill=lightgray] {};
            \vertex (7) at (1,-3.6) [scale=.75] {};
            \vertex (8) at (1,-2.8) [scale=.75] {};

            \path
            (1) edge (2)
            (2) edge (3)
            (3) edge (4)
            (4) edge (5)
            (5) edge (6)
            (6) edge (7)
            (7) edge (8)
            (2) edge (8)
            ;

            \vertex (12) at (-2,0) [scale=.75,fill=lightgray] {};
            \node ($y$) at (-1.5,-0.5) {$y$};
            \vertex (13) at (-2.6,1) [scale=.75] {};
            \vertex (14) at (-3.6,1) [scale=.75,fill=lightgray] {};
            \vertex (15) at (-4.6,1) [scale=.75,fill=black] {};
            \node ($y_{3}$) at (-4.6,1.7) {$y_{3}$};
            \vertex (16) at (-5.6,1) [scale=.75,fill=lightgray] {};
            \vertex (17) at (-6.6,0.5) [scale=.75] {};
            \vertex (18) at (-6.6,-0.5) [scale=.75] {};
            \vertex (19) at (-5.6,-1) [scale=.75,fill=lightgray] {};
            \vertex (20) at (-4.6,-1) [scale=.75,fill=black] {};
            \node ($y_{3}'$) at (-4.6,-1.8) {$y_{3}'$};
            \vertex (21) at (-3.6,-1) [scale=.75,fill=lightgray] {};
            \vertex (22) at (-2.6,-1) [scale=.75] {};

            \path
            (1) edge (12)
            (12) edge (13)
            (13) edge (14)
            (14) edge (15)
            (15) edge (16)
            (16) edge (17)
            (17) edge (18)
            (18) edge (19)
            (19) edge (20)
            (20) edge (21)
            (21) edge (22)
            (22) edge (12)
            ;

            \vertex (9) at (-1.7,-1.7) [scale=.75,fill=lightgray] {};
            \vertex (10) at (-2.6,-2.6) [scale=.75] {};
            \vertex (11) at (-3.5,-3.5) [scale=.75] {};

            \path
            (1) edge (9)
            (9) edge (10)
            (10) edge (11)
            ;

            \vertex (23) at (-1.5,1.7) [scale=.75,fill=lightgray] {};
            \vertex (24) at (-2.9,2) [scale=.75] {};
            \vertex (25) at (-1.8,3) [scale=.75] {};

            \path
            (1) edge (23)
            (23) edge (24)
            (23) edge (25)
            ;

            \vertex (26) at (0,2) [scale=.75,fill=lightgray] {};
            \vertex (27) at (-0.7,3) [scale=.75] {};
            \vertex (28) at (0.7,3) [scale=.75] {};

            \path
            (1) edge (26)
            (26) edge (27)
            (27) edge (28)
            (26) edge (28)
            ;

            \vertex (29) at (2,0) [scale=.75,fill=lightgray] {};
            \node ($v$) at (1.7,-0.5) {$v$};
            \vertex (30) at (4.5,3) [scale=.75] {};
            \vertex (31) at (4.5,1.5) [scale=.75] {};
            \vertex (32) at (4.5,0) [scale=.75] {};
            \vertex (33) at (4.5,-0.85) [scale=.5,fill=black] {};
            \vertex (34) at (4.5,-1.5) [scale=.5,fill=black] {};
            \vertex (35) at (4.5,-2.15) [scale=.5,fill=black] {};
            \vertex (36) at (4.5,-3) [scale=.75] {};
            
            \path
            (29) edge (30)
            (29) edge (31)
            (29) edge (32)
            (29) edge (36)
            (1)  edge (29)
            ;

            \vertex (37) at (1.48,1.58) [scale=.01,fill=black] {};
            \vertex (38) at (5.75,3) [scale=.02] {};
            \vertex (39) at (5.75,1.5) [scale=.02] {};
            \vertex (40) at (5.75,0) [scale=.02] {};
            \vertex (41) at (5.75,-3) [scale=.02] {};

            \path
            (30) edge (38) [dashed]
            (31) edge (39)
            (32) edge (40)
            (36) edge (41)
            (1)  edge (37)
            ;
        \end{tikzpicture}
    \end{center}
    \vspace{1.5mm}
    \par {\footnotesize \centerline{{\bf Fig. 2.} ~The case that $\mathcal{H}_b^x \neq \emptyset$.\hypertarget{Fig2}}}
    \end{figure}

It is easy to see that each component of $\mathcal{H}_g^x$ is not an $\mathcal{S}$-graph. We distinguish the following into two subcases.

    \emph{Subcase 1.1.} $G_v$ is not an $\mathcal{S}$-graph. Since $E(X\setminus N[D_X],V \setminus X)=\emptyset$, we have $\iota_1(G) \leq |D_X|+\iota_1(G-X)$ by Lemma \ref{Lemma2.2}. Note that each component of $G-X$ is not an $\mathcal{S}$-graph and contains no 6-cycles. By Lemma \ref{Lemma2.3} and by the induction hypothesis, we have
    $$\begin{aligned}
        \iota_1(G) &\leq |D_X|+\iota_1(G-X) \leq (1+k_7+2k_{11}) + \frac{1}{4}|V(G-X)|\\
                   &= (1+k_7+2k_{11})+\frac{1}{4}(n-1-3k_3-7k_7-11k_{11})\\
                   &= \frac{n}{4}+\frac{3}{4} (1-(k_3+k_7+k_{11})) \leq \frac{n}{4}.
    \end{aligned}$$

    \emph{Subcase 1.2.} $G_v$ is an $\mathcal{S}$-graph. Let $Y=X \cup V(G_v)$. Then, $G-Y= \bigcup_{H \in \mathcal{H}_g^x} H$. If $G_v \in \{C_7, C_{11}\}$, then let $v_3$, $v_3'$ be the two vertices of $G_v$ which are distance 3 from $v$ on the cycle.

    \emph{Subcase 1.2.1.} $G_v \in \{P_3, C_3\}$. Recall that $N(H) = \{x\}$ for each $H \in \mathcal{H}_b^x$. Clearly, $D_X$ is also a 1-isolating set of $G[Y]$. Since  $E(Y \setminus N[D_X], V \setminus Y) = \emptyset$, $\iota_1(G) \leq |D_X|+\iota_1(G-Y)$ by Lemma \ref{Lemma2.2}. Note that each component of $G-Y$ is not an $\mathcal{S}$-graph and contains no 6-cycles. By Lemma \ref{Lemma2.3} and by the induction hypothesis, we have
    $$\begin{aligned}
        \iota_1(G) &\leq |D_X|+\iota_1(G-Y) = (1+k_7+2k_{11})+ \sum_{H \in \mathcal{H}_g^x} \iota_1(H)\\
                   &\leq (1+k_7+2k_{11})+\frac{1}{4} (n-1-3(k_3+1)-7k_7-11k_{11})\\
                   &= \frac{n}{4}-\frac{3}{4}(k_3+k_7+k_{11}) < \frac{n}{4}.
    \end{aligned}$$

    \emph{Subcase 1.2.2.} $G_v \cong C_7$. 
    Clearly, $\{v_3\} \cup D_X$ is a 1-isolating set of $G[Y]$. Hence, by Lemmas \ref{Lemma2.2} and \ref{Lemma2.3}, and by the induction hypothesis, we have
    $$\begin{aligned}
        \iota_1(G)& \leq |\{v_3\} \cup D_X|+\iota_1(G-Y) = (2+k_7+2k_{11})+ \sum_{H \in \mathcal{H}_g^x} \iota_1(H)\\
                  & \leq (2+k_7+2k_{11})+\frac{1}{4} (n-1-3k_3-7(k_7+1)-11k_{11})\\
                  &=\frac{n}{4}-\frac{3}{4}(k_3+k_7+k_{11}) < \frac{n}{4}.
    \end{aligned}$$

    \emph{Subcase 1.2.3.} $G_v \cong C_{11}$. 
    Clearly, $\{v_3, v_3'\} \cup D_X$ is a 1-isolating set of $G[Y]$. Hence, by Lemmas \ref{Lemma2.2} and \ref{Lemma2.3}, and by the induction hypothesis, we have
    $$\begin{aligned}
        \iota_1(G)& \leq |\{v_3, v_3'\} \cup D_X|+\iota_1(G-Y) = (3+k_7+2k_{11})+ \sum_{H \in \mathcal{H}_g^x} \iota_1(H)\\
                  & \leq (3+k_7+2k_{11})+\frac{1}{4} (n-1-3k_3-7k_7-11(k_{11}+1))\\
                  &=\frac{n}{4}-\frac{3}{4}(k_3+k_7+k_{11}) < \frac{n}{4}.
    \end{aligned}$$

\vspace{3mm}

{\bf Case 2.} For any $x \in N(v)$, $\mathcal{H}_b^x = \emptyset$.

\vspace{3mm}

Now we fix a vertex $x\in N(v)$ with the property that there exists some $H^* \in \mathcal{H}_b$ with $x \in N(H^*)$. Let $X=V(H^*)  \cup  \{x\}$.  Then $G-X= G_v\cup \bigcup_{H \in \mathcal{H}_g^x}H$, where $G_v$ is the component of $G-X$ containing $v$. Clearly, $N(v) \setminus \{x\} \subseteq V(G_v)$.

\emph{Subcase 2.1.} $G_v$ is not an $\mathcal{S}$-graph. Let $xy \in E(G)$ for some $y \in V(H^*)$, and if $H^* \in \{C_7,C_{11}\}$, then let $y_{d}, y_{d}'$ be the two vertices of $H^*$ which are distance $d$ from $y$ on the cycle. If $H^* \cong P_3$ and $d_{H^*}(y) = 1$, then let $y_d$ be the vertex of $H^*$ with distance $d$ from $y$. If $x'y_{2} \in E(G)$ or $x'y_{2}' \in E(G)$ for some $x' \in N(v) \setminus \{x\}$, then $yy_{1}y_{2}x'vxy$ or $yy_{1}'y_{2}'x'vxy$ is a 6-cycle in $G$, a contradiction. So, $x'y_{2} \notin E(G)$ and $x'y_{2}' \notin E(G)$. Take
\begin{displaymath}
    D_X = \left\{
    \begin{array}{ll}
    \{y\},& \text{if}\ H^* \in \{P_3, C_3\},\\
   \{y, y_3\},& \text{if}\ H^* \cong C_{7},\\
   \{y, y_{4}, y_{4}'\},& \text{if}\ H^* \cong C_{11}.\\
    \end{array} \right.
\end{displaymath}
Clearly, $D_X$ is a 1-isolating set of $G[X]$, as shown in Fig. \hyperlink{Fig3}{3}. Since $E(X\setminus N[D_X],V \setminus X) =\emptyset$, then $\iota_1(G) \leq |D_X|+\iota_1(G-X)$ by Lemma \ref{Lemma2.2}. Note that each component of $G-X$ is not an $\mathcal{S}$-graph and contains no 6-cycles. By Lemma \ref{Lemma2.3} and by the induction hypothesis, we have
$$\begin{aligned}
    \iota_1(G)& \leq |D_X|+\iota_1(G-X) = \frac{1}{4} (|V(H^*)|+1)+ \iota_1(G_v)+\sum_{H \in \mathcal{H}_g^x} \iota_1(H)\\
              & \leq \frac{1}{4} (|V(H^*)|+1)+\frac{1}{4} (n-|V(H^*)|-1) \leq \frac{n}{4}.              
 \end{aligned}$$

\begin{figure}[h!]
    \begin{center}
    \begin{tikzpicture}[scale=.48]
    \tikzstyle{vertex}=[circle, draw, inner sep=0pt, minimum size=6pt]
    \tikzset{vertexStyle/.append style={rectangle}}
        \vertex (1) at (0,0) [scale=.75,fill=lightgray] {};
        \node ($x$) at (0.7,-0.5) {$x$};
        \vertex (2) at (0,-2) [scale=.75,fill=black] {};
        \vertex (3) at (-1,-2.8) [scale=.75,fill=lightgray] {};
        \vertex (4) at (-1,-3.6) [scale=.75,fill=lightgray] {};
        \vertex (5) at (-0.5,-4.4) [scale=.75,fill=black] {};
        \vertex (6) at (0.5,-4.4) [scale=.75,fill=lightgray] {};
        \vertex (7) at (1,-3.6) [scale=.75] {};
        \vertex (8) at (1,-2.8) [scale=.75,fill=lightgray] {};

        \path
        (1) edge (2)
        (2) edge (3)
        (3) edge (4)
        (4) edge (5)
        (5) edge (6)
        (6) edge (7)
        (7) edge (8)
        (2) edge (8)
        ;
       
        \vertex (12) at (-2,0) [scale=.75,fill=black] {};
        \node ($y$) at (-1.6,-0.5) {$y$};
        \vertex (13) at (-2.6,1) [scale=.75,fill=lightgray] {};
        \vertex (14) at (-3.6,1) [scale=.75] {};
        \vertex (15) at (-4.6,1) [scale=.75,fill=lightgray] {};
        \vertex (16) at (-5.6,1) [scale=.75,fill=black] {};
        \node ($y_4$) at (-5.8,1.7) {$y_4$};
        \vertex (17) at (-6.6,0.5) [scale=.75,fill=lightgray] {};
        \vertex (18) at (-6.6,-0.5) [scale=.75,fill=lightgray] {};
        \vertex (19) at (-5.6,-1) [scale=.75,fill=black] {};
        \node ($y_{4}'$) at (-5.8,-1.8) {$y_{4}'$};
        \vertex (20) at (-4.6,-1) [scale=.75,fill=lightgray] {};
        \vertex (21) at (-3.6,-1) [scale=.75] {};
        \vertex (22) at (-2.6,-1) [scale=.75,fill=lightgray] {};

        \path
        (1) edge (12)
        (12) edge (13)
        (13) edge (14)
        (14) edge (15)
        (15) edge (16)
        (16) edge (17)
        (17) edge (18)
        (18) edge (19)
        (19) edge (20)
        (20) edge (21)
        (21) edge (22)
        (22) edge (12)
        ;

        \vertex (9) at (-1.7,-1.7) [scale=.75,fill=black] {};
        \vertex (10) at (-2.6,-2.6) [scale=.75,fill=lightgray] {};
        \vertex (11) at (-3.5,-3.5) [scale=.75] {};

        \path
        (1) edge (9)
        (9) edge (10)
        (10) edge (11)
        ;

        \vertex (23) at (-1.5,1.7) [scale=.75,fill=black] {};
        \vertex (24) at (-2.9,2) [scale=.75,fill=lightgray] {};
        \vertex (25) at (-1.8,3) [scale=.75,fill=lightgray] {};

        \path
        (1) edge (23)
        (23) edge (24)
        (23) edge (25)
        ;

        \vertex (26) at (0,2) [scale=.75,fill=black] {};
        \vertex (27) at (-0.7,3) [scale=.75,fill=lightgray] {};
        \vertex (28) at (0.7,3) [scale=.75,fill=lightgray] {};

        \path
        (1) edge (26)
        (26) edge (27)
        (27) edge (28)
        (26) edge (28)
        ;

        \vertex (29) at (2,0) [scale=.75] {};
        \node ($v$) at (1.7,-0.5) {$v$};
        \vertex (30) at (4.5,3) [scale=.75] {};
        \vertex (31) at (4.5,1.5) [scale=.75] {};
        \vertex (32) at (4.5,0) [scale=.75] {};
        \vertex (33) at (4.5,-0.85) [scale=.5,fill=black] {};
        \vertex (34) at (4.5,-1.5) [scale=.5,fill=black] {};
        \vertex (35) at (4.5,-2.15) [scale=.5,fill=black] {};
        \vertex (36) at (4.5,-3) [scale=.75] {};
        
        \path
        (29) edge (30)
        (29) edge (31)
        (29) edge (32)
        (29) edge (36)
        (1)  edge (29)
        ;

        \vertex (37) at (0.7,-1.55) [scale=.01,fill=black] {};
        \vertex (38) at (5.75,3) [scale=.02] {};
        \vertex (39) at (5.75,1.5) [scale=.02] {};
        \vertex (40) at (5.75,0) [scale=.02] {};
        \vertex (41) at (5.75,-3) [scale=.02] {};

        \path
        (30) edge (38) [dashed]
        (31) edge (39)
        (32) edge (40)
        (36) edge (41)
        (1)  edge (37)
        (8) edge (36)
        (26) edge (30)
        ;

        \draw (-1.7,3.1) parabola bend (1.1,3.9) (4.3,3) [dashed];
        \draw (-5.5,1.1) parabola bend (-0.3,4.6) (4.3,3.15) [dashed];
        \draw (-2.5,-2.8) parabola bend (0.2,-5.4) (4.4,-3.2) [dashed];
    \end{tikzpicture}
\end{center}
\vspace{1.5mm}
\par {\footnotesize \centerline{{\bf Fig. 3.} ~The 1-isolating set $D_X$ of $G[X]$. \hypertarget{Fig3}}}
\end{figure}

\emph{Subcase 2.2.}  $G_v$ is an $\mathcal{S}$-graph. It follows that $H^*$ is the only component of $\mathcal{H}_b$ with $x \in N(H^*)$. Since $d_{G_v}(v)=2$, $\Delta(G)=d_G(v)=d_{G_v}(v)+|\{x\}|=3$. Let $Y=X \cup V(G_v)$. Then $G-Y = \bigcup_{H \in \mathcal{H}_g^x}H$.

\vspace{3mm}

\begin{claim} \label{ClaimB}
     $\mathcal{H}_g^x = \emptyset$. 
\end{claim}

\begin{proof}
On the contrary, suppose $\mathcal{H}_g^x \neq \emptyset$. Since $\{v,y\} \in N(x)$, $d(x) = \Delta(G)=3$. Let $H'$ be the only component of $\mathcal{H}_g^x$, where $N(H')=\{x\}$. Let $xz \in E(G)$ for some $z \in V(H')$, as shown in Fig. \hyperlink{Fig4}{4}.

\begin{figure}[h!]
    \begin{center}
    \begin{tikzpicture}[scale=.48]
    \tikzstyle{vertex}=[circle, draw, inner sep=0pt, minimum size=6pt]
    \tikzset{vertexStyle/.append style={rectangle}}
        \vertex (1) at (0,0) [scale=.75] {};
        \node ($x$) at (-0.5,-0.5) {$x$};
        \vertex (2) at (0,3) [scale=.75] {};
        \node ($y$) at (-0.5,2.5) {$y$};
        \vertex (3) at (-1,4) [scale=.75] {};
        \vertex (4) at (1,4) [scale=.75] {};
        \vertex (5) at (-1,5) [scale=.05] {};
        \vertex (6) at (1,5) [scale=.05] {};
        
        \path
        (1) edge (2)
        (2) edge (3)
        (4) edge (2)
        ;

        \path
        (5) edge (3)[dashed]
        (6) edge (4)
        ;

        \vertex (7) at (3,0) [scale=.75] {};
        \node ($v$) at (2.6,-0.5) {$v$};
        \vertex (8) at (4,1) [scale=.75] {};
        \vertex (9) at (4,-1) [scale=.75] {};
        \vertex (10) at (5.25,1) [scale=.02] {};
        \vertex (11) at (5.25,-1) [scale=.02] {};

        \path
        (1) edge (7)
        (7) edge (8)
        (7) edge (9)
        ;

        \path
        (8) edge (10)[dashed]
        (9) edge (11)
        (4) edge (8)
        ;
       
        \vertex (12) at (-3,0) [scale=.75] {};
        \node ($z$) at (-2.6,-0.5) {$z$};
        \vertex (13) at (-3.9,0.9) [scale=.02] {};
        \vertex (14) at (-3.9,-0.9) [scale=.02] {};

        \path
        (12) edge (13)[dashed]
        (12) edge (14)
        ;

        \path
        (1) edge (12)
        ;

        \draw[dashed] (-1.6,5.7) rectangle (5.8,-1.7);

        \node ($G[Y]$) at (4.3,4.6) {$G[Y]$};
    \end{tikzpicture}
\end{center}
\vspace{1.5mm}
\par {\footnotesize \centerline{{\bf Fig. 4.} ~For the case that $\mathcal{H}_g^x \neq \emptyset$. \hypertarget{Fig4}}}
\end{figure}

Since $|N(H^*)| \geq 2$, $G[Y]-x=G [V(H^*) \cup V(G_v)]$ is connected. Since $H^*$ and $G_v$ are $\mathcal{S}$-graphs, $|V(H^*)| \in \{3,7,11\}$ and $|V(G_v)| \in \{3,7,11\}$. Clearly, $|V(H^*) \cup V(G_v)|=|V(H^*)|+| V(G_v)| \in \{6,10,14,18,22\}$. That is, $G[Y]-x$ must not be an $\mathcal{S}$-graph. Set $Z = Y \cup \{z\}$. By the induction hypothesis, we have $$\iota_1 (G[Y]-x) \leq \lfloor \frac{1}{4} (|V(H^*)|+|V(G_v)|) \rfloor= \frac{1}{4} (|V(H^*)|+|V(G_v)|-2)=\frac{1}{4}(|Z|-4).$$

Let $D_{Y \setminus \{x\}}$ be a 1-isolating set of $G[Y]-x$ of size $\iota_1 (G[Y]-x)$. Then, $D_{Z}=D_{Y \setminus \{x\}} \cup \{x\}$ is a 1-isolating set of $G[Z]$. Furthermore, $$|D_Z| = |D_{Y \setminus \{x\}} \cup \{x\}| = \iota_1 (G[Y]-x)+|\{x\}| \leq \frac{1}{4}(|Z|-4)+1=\frac{1}{4}|Z|.$$

Since $d(z) \leq \Delta(G)=3$, $G-Z$ has at most two components. If any component of $G-Z$ is not an $\mathcal{S}$-graph, then by the induction hypothesis, we have $\iota_1(G-Z) \leq \frac{1}{4}|V(G-Z)|$. Note that $z \in N[D_{Z}]$, and $E(Z\setminus N[D_{Z}], V\setminus Z) =\emptyset$. Hence, by Lemma \ref{Lemma2.2}, $\iota_1(G)\leq |D_{Z}|+\iota_1(G-Z) \leq \frac{1}{4}(|Z|+|V(G-Z)|)=\frac{n}{4}$. If $G-Z$ has a component isomorphic to an $\mathcal{S}$-graph, then it is easy to see $d(x)=\Delta(G)=3$ and $z \in N(x)$. Clearly, $\mathcal{H}_b^z \neq \emptyset$, and we return to Case 1. This proves Claim \ref{ClaimB}.
\end{proof}

By Claim \ref{ClaimB}, $\mathcal{H}_g^x = \emptyset$. Recall that $X = V(H^*) \cup \{x\}$. Then $G-X = G_v$. Let $xy \in E(G)$ for some $y \in V(H^*)$. If $G_v \cong P_3$, then since $d(v)=\Delta(G)=3$, $d_{G_v}(v)=2$. Let $v_d$, $v_d'$ be the two vertices distance $d$ from $v$ in $G_v$. Let $y_d$, $y_d'$ be the two vertices distance $d$ from $y$ in $H^*$ if $H^* \in \{C_3,C_7,C_{11}\}$.

\vspace{3mm}

\emph{Subcase 2.2.1.} $H^* \cong C_3$.  Since $|N(H^*)| \geq 2$ and $d(y)=\Delta(G)=3$, we may assume that $y_1v_1 \in E(G)$. Then $y_1y_1'yxvv_1y_1$ is a 6-cycle in $G$, a contradiction.

\emph{Subcase 2.2.2.} $H^* \cong P_3$. We need to consider the degree of $y$ in $H^*$. 

(I) $d_{H^*}(y)=2$. Then $d(y)=3$. Let $N_{H^*}(y)=\{y_1,y_1'\}$. We may assume that $y_1v_1 \in E(G)$. However, $y_1yxv_1'vv_1y_1$ is a 6-cycle in $G$ if $xv_1' \in E(G)$, $y_1yy_1'v_1'vv_1y_1$ is a 6-cycle in $G$ if $y_1'v_1' \in E(G)$, and $y_1yy_1'xvv_1y_1$ is a 6-cycle in $G$ if $y_1'x \in E(G)$. So, $xv_1',y_1'v_1',y_1'x \notin E(G)$. We now consider the structure of $G_v$.

(i) $G_v \cong C_3$. It is easy to see that $y_1yxvv_1'v_1y_1$ is a 6-cycle in $G$, a contradiction.

(ii) $G_v \cong P_3$. Since $d(v) = \Delta(G)=3$, $d_{G_v}(v)=2$. Take
\begin{displaymath}
        D = \left\{
        \begin{array}{ll}
        \{y_1\},& \text{if}\ y_1v_1' \in E(G),\\
        \{v_1\},& \text{if}\ y_1'v_1 \in E(G),\\
        \{x\},&  \text{otherwise}.\\
        \end{array} \right.
    \end{displaymath}
    Clearly, $D$ is a 1-isolating set of $G$. Hence, $\iota_1(G) \leq |D|= 1< \frac{7}{4}=\frac{n}{4}$.

(iii) $G_v \cong C_7$. Since $d_G(v_1)=\Delta(G)=3$, $y_1'v_1 \notin E(G)$. Recalling that $y_1'v_1' \notin E(G)$ and $y_1'x \notin E(G)$, we determine $d(y_1')=1$. As shown in Fig. \hyperlink{Fig5}{5}, take
    \begin{displaymath}
        D = \left\{
        \begin{array}{ll}
        \{x,v_3\},& \text{if}\ y_1v_1' \notin E(G),\\
        \{y_1,v_3\},& \text{if}\ y_1v_1' \in E(G), \ xv_2' \notin E(G),\\
        \{v_1, v_2'\},& \text{if}\ y_1v_1' \in E(G), \ xv_2' \in E(G).\\
        \end{array} \right.
    \end{displaymath}
    Clearly, $D$ is a 1-isolating set of $G$. Hence, $\iota_1(G) \leq |D|= 2< \frac{11}{4}=\frac{n}{4}$.

    \begin{figure}[h!]
        \begin{center}
        \begin{tikzpicture}[scale=.48]
        \tikzstyle{vertex}=[circle, draw, inner sep=0pt, minimum size=6pt]
        \tikzset{vertexStyle/.append style={rectangle}}

%Center
        \vertex (1) at (0.5,0) [scale=.55] {};
        \node [scale=.85] ($x$) at (0.5,0.5) {$x$} ;
        \vertex (2) at (2,0) [scale=.55] {};
        \node ($v$) [scale=.85] at (1.93,0.5) {$v$};
        \vertex (3) at (2.8,-1) [scale=.55,fill=lightgray] {};
        \node ($v_1'$) [scale=.85] at (2.8,-1.7) {$v_1'$};
        \vertex (4) at (3.8,-1) [scale=.55] {};
        \node ($v_2'$) [scale=.85] at (4,-1.7) {$v_2'$};
        \vertex (5) at (4.8,-0.5) [scale=.55,fill=lightgray] {};
        \node ($v_3'$) [scale=.85] at (5.3,-1) {$v_3'$};
        \vertex (6) at (4.8,0.5) [scale=.55,fill=black] {};
        \node ($v_3$) [scale=.85] at (5.3,0.8) {$v_3$};
        \vertex (7) at (3.8,1) [scale=.55,fill=lightgray] {};
        \node ($v_2$) [scale=.85] at (4,1.5) {$v_2$};
        \vertex (8) at (2.8,1) [scale=.55,fill=lightgray] {};
        \node ($v_1$) [scale=.85] at (2.8,1.5) {$v_1$};

        \path
        (1) edge (2)
        (2) edge (3)
        (3) edge (4)
        (4) edge (5)
        (5) edge (6)
        (6) edge (7)
        (7) edge (8)
        (8) edge (2)
        ;

        \vertex (9) at (-1,0) [scale=.55,fill=lightgray] {};
        \node ($y$) [scale=.85] at (-0.7,0.5) {$y$};
        \vertex (10) at (-2,1) [scale=.55,fill=black] {};
        \node ($y_1$) [scale=.85] at (-2.5,1.5) {$y_1$};
        \vertex (11) at (-2,-1) [scale=.55] {};
        \node ($y_1'$) [scale=.85] at (-1.3,-1.1) {$y_1'$};

        \path
        (1) edge (9)
        (9) edge (10)
        (9) edge (11)
        ;

        \draw (2.8,1.15)..controls (2,1.8) and (-1.3,1.8) ..(-1.9,1.15);
        \draw (-1.9,1.15)..controls (-4.6,-3.6) and (2.5,-1.5) ..(2.7,-1.1) ;

        %\node (1) [scale=.85] at (1,-3.5) {If $y_1v_1' \in E(G)$, $xv_2' \notin E(G)$};

%Left
        \vertex (21) at (-9.5,0) [scale=.55,fill=black] {};
        \node ($x$) [scale=.85] at (-9.5,0.5) {$x$};
        \vertex (22) at (-8,0) [scale=.55,fill=lightgray] {};
        \node ($v$) [scale=.85] at (-8.07,0.5) {$v$};
        \vertex (23) at (-7.2,-1) [scale=.55] {};
        \node ($v_1'$) [scale=.85] at (-7.2,-1.7) {$v_1'$};
        \vertex (24) at (-6.2,-1) [scale=.75] {};
        \node ($v_2'$) [scale=.85] at (-6,-1.7) {$v_2'$};
        \vertex (25) at (-5.2,-0.5) [scale=.55,fill=lightgray] {};
        \node ($v_3'$) [scale=.85] at (-4.7,-1) {$v_3'$};
        \vertex (26) at (-5.2,0.5) [scale=.55,fill=black] {};
        \node ($v_3$) [scale=.85] at (-4.7,0.8) {$v_3$};
        \vertex (27) at (-6.2,1) [scale=.55,fill=lightgray] {};
        \node ($v_2$) [scale=.85] at (-6,1.5) {$v_2$};
        \vertex (28) at (-7.2,1) [scale=.55] {};
        \node ($v_1$) [scale=.85] at (-7.2,1.5) {$v_1$};

        \path
        (21) edge (22)
        (22) edge (23)
        (23) edge (24)
        (24) edge (25)
        (25) edge (26)
        (26) edge (27)
        (27) edge (28)
        (28) edge (22)
        ;

        \vertex (29) at (-11,0) [scale=.55,fill=lightgray] {};
        \node ($y$) [scale=.85] at (-10.7,0.5) {$y$};
        \vertex (30) at (-12,1) [scale=.55] {};
        \node ($y_1$) [scale=.85] at (-12.5,1.5) {$y_1$};
        \vertex (31) at (-12,-1) [scale=.55] {};
        \node ($y_1'$) [scale=.85] at (-12.5,-1.5) {$y_1'$};

        \path
        (21) edge (29)
        (29) edge (30)
        (29) edge (31)
        ;

        \draw (-7.2,1.15)..controls (-8,1.8) and (-11.3,1.8) ..(-12,1.15);

        %\node (2) [scale=.85] at (-9,-3.5) {If $y_1v_1' \notin E(G)$};

    %Right
        \vertex (41) at (10.5,0) [scale=.55,fill=lightgray] {};
        \node ($x$) [scale=.85] at (10.5,0.5) {$x$};
        \vertex (42) at (12,0) [scale=.55,fill=lightgray] {};
        \node ($v$) [scale=.85] at (11.93,0.5) {$v$};
        \vertex (43) at (12.8,-1) [scale=.55,fill=lightgray] {};
        \node ($v_1'$) [scale=.85] at (12.1,-1) {$v_1'$};
        \vertex (44) at (13.8,-1) [scale=.55,fill=black] {};
        \node ($v_2'$) [scale=.85] at (14,-1.7) {$v_2'$};
        \vertex (45) at (14.8,-0.5) [scale=.55,fill=lightgray] {};
        \node ($v_3'$) [scale=.85] at (15.3,-1) {$v_3'$};
        \vertex (46) at (14.8,0.5) [scale=.55] {};
        \node ($v_3$) [scale=.85] at (15.3,0.8) {$v_3$};
        \vertex (47) at (13.8,1) [scale=.55,fill=lightgray] {};
        \node ($v_2$) [scale=.85] at (14,1.5) {$v_2$};
        \vertex (48) at (12.8,1) [scale=.55,fill=black] {};
        \node ($v_1$) [scale=.85] at (12.8,1.5) {$v_1$};

        \path
        (41) edge (42)
        (42) edge (43)
        (43) edge (44)
        (44) edge (45)
        (45) edge (46)
        (46) edge (47)
        (47) edge (48)
        (48) edge (42)
        ;

        \vertex (49) at (9,0) [scale=.55] {};
        \node ($y$) [scale=.85] at (9.3,0.5) {$y$};
        \vertex (50) at (8,1) [scale=.55,fill=lightgray] {};
        \node ($y_1$) [scale=.85] at (7.5,1.5) {$y_1$};
        \vertex (51) at (8,-1) [scale=.55] {};
        \node ($y_1'$) [scale=.85] at (8.7,-1.1) {$y_1'$};

        \path
        (41) edge (49)
        (49) edge (50)
        (49) edge (51)
        ;

        \draw (12.7,1.13)..controls (12,1.8) and (8.7,1.8) ..(8.1,1.13);
        \draw (7.9,0.87)..controls (5.4,-3.6) and (12.1,-2.5) ..(12.7,-1.1);
        \draw (10.5,-0.12)..controls (11,-2) and (13,-2) ..(13.8,-1) ;

        %\node (3) [scale=.85] at (12,-3.5) {If $y_1v_1' \in E(G)$, $xv_2' \in E(G)$};

        \end{tikzpicture}
    \end{center}
    \par {\footnotesize \centerline{{\bf Fig. 5.} ~For the subcases that $H^* \cong P_3$ and $G_v \cong C_7$. \hypertarget{Fig5}}}
    \end{figure}

    (iv) $G_v \cong C_{11}$. Since $d_G(v_1)=\Delta(G)=3$, $y_1'v_1 \notin E(G)$. Recalling that $y_1'v_1' \notin E(G)$ and $y_1'x \notin E(G)$, we determine $d(y_1')=1$. Take
    \begin{displaymath}
        D = \left\{
        \begin{array}{ll}
        \{x,v_3, v_3'\},& \text{if}\ y_1v_1' \notin E(G),\\
        \{y_1,v_3, v_3'\},& \text{if}\ y_1v_1' \in E(G), \ xv_5 \notin E(G) \ \text{and}\ xv_5' \notin E(G),\\
        \{x, v_2,v_2'\},& \text{if}\ y_1v_1' \in E(G), \ xv_5 \in E(G) \ \text{or}\ xv_5' \in E(G).\\
        \end{array} \right.
    \end{displaymath}
    Clearly, $D$ is a 1-isolating set of $G$. Hence, $\iota_1(G) \leq |D|= 3< \frac{17}{4}=\frac{n}{4}$.

    (II) $d_{H^*}(y)=1$. Let $N_{H^*}(y_1)=\{y,y_2\}$. If $y_2v' \in E(G)$ for some $v' \in \{v_1,v_1'\}$, then $yy_1y_2v'vxy$ is a 6-cycle in $G$, a contradiction. So, let $y_2v_1,y_2v_1' \notin E(G)$. If $yv_1 \in E(G)$, then take
    \begin{displaymath}
        D = \left\{
        \begin{array}{ll}
        \{y\},& \text{if}\ G_v\in \{P_3, C_3\},\\
        \{y, v_2'\},& \text{if}\ G_v\cong C_{7},\\
        \{y, v_3,v_3'\},& \text{if}\ G_v\cong C_{11}.\\
        \end{array} \right.
    \end{displaymath}
    Clearly, $D$ is a 1-isolating set of $G$. Hence, $\iota_1(G) \leq |D| = \frac{n-3}{4}<\frac{n}{4}$. 
    
    So, let $yv_1 \notin E(G)$. By the symmetry of $v_1$ and $v_1'$, let $yv_1' \notin E(G)$. Since $|N(H^*)|\geq 2$, we may assume that $y_1v_1 \in E(G)$. It follows that $G_v \ncong C_3$, otherwise $yxvv_1'v_1y_1y$ is a 6-cycle in $G$. If $xv_1' \in E(G)$, then $yxv_1'vv_1y_1y$ is a 6-cycle in $G$. If $G_v \in \{C_7,C_{11}\}$ and $xv_3 \in E(G)$, then $yxv_3v_2v_1y_1y$ is a 6-cycle in $G$. So, let $xv_1' \notin E(G)$ and $xv_3 \notin E(G)$. If $xy_2 \in E(G)$, then take 
    \begin{displaymath}
        D = \left\{
        \begin{array}{ll}
        \{x\},& \text{if}\ G_v \cong P_3,\\
        \{x, v_3\},& \text{if}\ G_v\cong C_{7},\\
        \{x, v_3,v_3'\},& \text{if}\ G_v\cong C_{11}.\\
        \end{array} \right.
    \end{displaymath}
    Clearly, $D$ is a 1-isolating set of $G$. Hence, $\iota_1(G) \leq |D| = \frac{n-3}{4}<\frac{n}{4}$.
    
    So, let $xy_2 \notin E(G)$. Now we take
    \begin{displaymath}
        D = \left\{
        \begin{array}{ll}
        \{v_1\},& \text{if}\ G_v \cong P_3,\\
        \{v_1,v_3'\},& \text{if}\ G_v\cong C_{7},\\
        \{v_1, v_3', v_5\},& \text{if}\ G_v\cong C_{11}.\\
        \end{array} \right.
    \end{displaymath}
    Clearly, $D$ is a 1-isolating set of $G$. Hence, $\iota_1(G) \leq |D| = \frac{n-3}{4}<\frac{n}{4}$.

\emph{Subcase 2.2.3.}  $H^* \cong C_7$. We consider the structure of $G_v$.

    (i) $G_v \cong C_3$. Since $G$ contains no 6-cycles, $E(\{y_1, y_2, y_1',y_2'\},\{v_1,v_1'\})=\emptyset$. Hence, $D=\{x, y_3\}$ is a 1-isolating set of $G$, and $\iota_1(G) \leq |D|=2 < \frac{11}{4} =\frac{n}{4}$.

    (ii) $G_v \cong P_3$. For each $v' \in \{v_1,v_1'\}$, $yy_1y_2v'vxy$ is a 6-cycle in $G$ if $y_2v' \in E(G)$, and $yy_1'y_2'v'vxy$ is a 6-cycle in $G$ if $y_2'v' \in E(G)$. So, let $E(\{y_2,y_2'\},\{v_1,v_1'\}) =\emptyset$. If $E(\{y_1,y_1'\},\{v_1,v_1'\})=\emptyset$, then $D=\{x, y_3\}$ is a 1-isolating set of $G$, and then $\iota_1(G) \leq |D|=2 < \frac{11}{4} =\frac{n}{4}$.
    
    Hence, we may assume that $y_1v_1 \in E(G)$. It is noted that if $y_1'v_1' \in E(G)$, then $y_1yy_1'v_1'vv_1y_1$ is a 6-cycle in $G$, a contradiction. So, $y_1'v_1' \notin E(G)$. Take 
    \begin{displaymath}
        D = \left\{
        \begin{array}{ll}
        \{x, y_3\},& \text{if}\ \ y_1'v_1 \notin E(G),\\
        \{v_1, y_3'\},& \text{if}\ \ y_1'v_1 \in E(G),\ y_2x \notin E(G),\\
        \{y_1',y_2\},& \text{if}\ \ y_1'v_1 \in E(G),\ y_2x \in E(G).\\
        \end{array} \right.
    \end{displaymath}
    If $y_1'v_1 \in E(G)$ and $ y_2x \notin E(G)$, then $xv_1' \notin E(G)$. Otherwise, $xv_1'vv_1y_1yx$ is a 6-cycle in $G$. If $y_1'v_1 \in E(G)$ and $y_2x \in E(G)$, then $y_3'v_1' \notin E(G)$. Otherwise, $y_3'v_1'vxy_2y_3y_3'$ is a 6-cycle in $G$. Hence, $D$ is a 1-isolating set of $G$, and we have $\iota_1(G) \leq |D|=2 < \frac{11}{4} =\frac{n}{4}$.
    
    (iii) $G_v \cong C_7$. Since $G$ contains no 6-cycles, $E(\{y_2,y_2'\},\{v_1,v_1'\})=\emptyset$. If $E(\{y_1,y_1'\},\{v_1,v_1'\})=\emptyset$, then $D=\{x, y_3,v_3\}$ is a 1-isolating set of $G$. Assume that $E(\{y_1,y_1'\},\{v_1,v_1'\}) \neq \emptyset$ and $y_1v_1 \in E(G)$. Since $d(y_1)=d(v_1)=\Delta(G)=3$, 
    $y_1v_1' \notin E(G)$ and $ y_1'v_1 \notin E(G)$. Furthermore, $y_1'v_1' \notin E(G)$, otherwise $y_1'yy_1v_1vv_1'y_1'$ is a 6-cycle in $G$. It is noted that $D=\{x, y_3, v_3\}$ is also a 1-isolating set of $G$. Hence, $\iota_1(G) \leq |D|=3 < \frac{15}{4} =\frac{n}{4}$. 
   
    (iv) $G_v \cong C_{11}$. Similar to the subcase (iii), we know that $D= \{x, y_3, v_3,v_3'\}$ or $D=\{x,y_3',v_3,v_3'\}$ is a 1-isolating set of $G$, and we have $\iota_1(G) \leq |D|=4 < \frac{19}{4} =\frac{n}{4}$.

\emph{Subcase 2.2.4.} $H^* \cong C_{11}$. Since $G$ contains no 6-cycles, $E(\{y_2, y_2'\},\{v_1, v_1'\}) = \emptyset$. We further consider the structure of $G_v$.

(i) $G_v \cong C_3$. Since $G$ contains no 6-cycles, $E(\{y_1, y_1'\},\{v_1, v_1'\})=\emptyset$. Hence, $D =\{x, y_4,y_4'\}$ is a 1-isolating set of $G$, implying that $\iota_1(G) \leq |D|=3 < \frac{15}{4}=\frac{n}{4}$.

(ii) $G_v \cong P_3$. If $E(\{y_1, y_1'\},\{v_1, v_1'\})=\emptyset$, then $D= \{x, y_4,y_4'\}$ is a 1-isolating set of $G$. Assume that $y_1v_1 \in E(G)$. However, $y_1'v_1'vv_1y_1yy_1'$ is a 6-cycle in $G$ if $y_1'v_1' \in E(G)$, $y_5v_1y_1y_2y_3y_4y_5$ is a 6-cycle in $G$ if $y_5v_1 \in E(G)$, and $xv_1'vv_1y_1yx$ is a 6-cycle in $G$ if $xv_1' \in E(G)$. So, $y_1'v_1', y_5v_1, xv_1' \notin E(G)$. Take 
\begin{displaymath}
    D = \left\{
    \begin{array}{ll}
    \{x,y_4',y_3\},& \text{if}\ y_1'v_1 \notin E(G),\\
    \{y_4',y_4,v_1\},& \text{if}\ y_1'v_1 \in E(G), \ y_2x \notin E(G) \ \text{and}\ y_2'x \notin E(G),\\
    \{y_4', y_3, v_1\},& \text{if}\ y_1'v_1 \in E(G), \ y_2x \in E(G),\\
    \{y_3', y_4, v_1\},& \text{if}\ y_1'v_1 \in E(G), \ y_2'x \in E(G).\\
    \end{array} \right.
\end{displaymath}
Clearly, $D$ is a 1-isolating set of $G$. Hence, $\iota_1(G) \leq |D|=3 < \frac{15}{4}=\frac{n}{4}$.

(iii) $G_v \cong C_7$. Since $G$ contains no 6-cycles, $E(\{y_2, y_2'\},\{v_1, v_1'\})=\emptyset$. If $E(\{y_1,y_1'\},\{v_1,v_1'\}) = \emptyset$, then $D= \{x, y_4,y_4', v_3\}$ is a 1-isolating set of $G$. Assume that $y_1v_1 \in E(G)$. Then, $y_1'v_1' \notin E(G)$. Take
\begin{displaymath}
    D = \left\{
    \begin{array}{ll}
    \{x,y_4',y_3,v_3\},& \text{if}\ y_5v_1' \notin E(G),\\
    \{y,y_3',y_5,v_3\},& \text{if}\ y_5v_1' \in E(G).\\
    \end{array} \right.
\end{displaymath}
Since $d(v_1)=\Delta(G)=3$, $y_3v_1 \notin E(G)$. It is noted that $D$ is a 1-isolating set of $G$. Hence, $\iota_1(G)\leq |D|=4 < \frac{19}{4} = \frac{n}{4}$.

(iv) $G_v \cong C_{11}$. Let $D$ be the set defined as in (iii). Now $D \cup \{v_3'\}$ is a 1-isolating set of $G$. Hence, $\iota_1(G)\leq |D \cup \{v_3'\}|=5 < \frac{23}{4} = \frac{n}{4}$.

\vspace{3mm}

This completes the proof of Theorem \ref{Theorem1.6}.

\section {\large Proof of Theorem \ref{Theorem1.7}}

In this section, we present a proof of Theorem \ref{Theorem1.7}, which is similar to Theorem \ref{Theorem1.6}'s. Recall the statement of Theorem \ref{Theorem1.7}.

\begin{theorembis}{Theorem1.7}
    If $G \notin \{P_3,C_3,C_7,C_{11}\}$ is a connected graph of order $n$ without induced 5- and 6-cycles, then $\iota_1(G) \leq \frac{n}{4}$.
\end{theorembis}

\begin{proof}
Let $G=(V,E)$ be a connected graph of order $n$ without induced 5- and 6-cycles. Suppose that $G$ is not an $\mathcal{S}$-graph, that is, $G \notin \{P_3,C_3,C_7,C_{11}\}$. The proof is by induction on $n$. It is easy to see that $\iota_1(G) \leq \frac{n}{4}$ for the graphs $G$ of order $n \leq 3$. Let $n \geq 4$. If $\Delta(G) \leq 2$, then $G$ is a path or a cycle. By Lemma \ref{Lemma2.1}, $\iota_1(G)\leq \frac{n}{4}$. Fix a vertex $v \in V(G)$ with $d(v)=\Delta(G)$. If $\Delta(G)=n-1$, then $\iota_1(G)\leq |\{v\}| = 1 \leq \frac{n}{4}$. Hence, we assume that $3 \leq \Delta(G) \leq n-2$.

Let $G'=G-N[v]$ with $|V(G')|=n'$. Since $\Delta(G)\leq n-2$, $n' \geq 1$. Let $\mathcal{H}$ be the set of components of $G'$,  $\mathcal{H}_b$ be the set of components of $G'$ isomorphic to an $\mathcal{S}$-graph, and  $\mathcal{H}_g=\mathcal{H} \setminus \mathcal{H}_b$. By the induction hypothesis, $\iota_1(H) \leq \frac{1}{4}|V(H)|$ for any component $H \in \mathcal{H}_g$. If $\mathcal{H}_b = \emptyset$, then $\mathcal{H} = \mathcal{H}_g \neq \emptyset$. By Lemmas \ref{Lemma2.2} and \ref{Lemma2.3}, we have $\iota_1(G) \leq |\{v\}| + \iota_1(G') = 1 + \sum_{H \in \mathcal{H}}\iota_1(H) \leq 1+\frac{1}{4}(n-\Delta(G)-1) \leq \frac{n}{4}$. Hence, we assume that $\mathcal{H}_b \neq \emptyset$ in the following.

For any $x \in N(v)$, let $\mathcal{H}_b^x$ be the set of components $H$ of $\mathcal{H}_b$ with $N(H) = \{x\}$. Note that the graph $G$ of Theorem \ref{Theorem1.6} contains no 6-cycles, while Theorem \ref{Theorem1.7} requires that $G$ contains no induced 5- and 6-cycles. Applying the same way of Case 1 in the proof of Theorem \ref{Theorem1.6}, it is easy to check that Theorem \ref{Theorem1.7} is true for the case that $\mathcal{H}_b^x \neq \emptyset$ for some $x \in N(v)$. Therefore, in the following, we assume that for any $x \in N(v)$, $\mathcal{H}_b^x = \emptyset$, equivalently, for any $H \in \mathcal{H}_b$, $|N(H)| \geq 2$.

Let $k_3$ be the number of components isomorphic to $P_3$ or $C_3$ in $\mathcal{H}_b$, $k_i$ be the number of components isomorphic to  $C_i$ in $\mathcal{H}_b$, where $i \in \{7,11\}$. By the present assumption, $|\mathcal{H}_b|=k_3+k_7+k_{11} \geq 1$. Denote $\Delta(G) = \Delta$ simply. %We prove the following by discussing $\Delta(G)$.

\vspace{3mm}

\begin{claim} \label{ClaimC}
    $|\mathcal{H}_b|+1 \leq \Delta \leq |\mathcal{H}_b|+2$.
\end{claim}

\begin{proof}
    Let $X=N[v] \cup \bigcup_{H \in \mathcal{H}_b}V(H)$. Then $G-X=\bigcup_{H \in \mathcal{H}_g}H$. For each $H \in \mathcal{H}_b$, let $xy \in E(G)$ for some $x \in N(v)$ and $y \in V(H)$, and if $H \in \{C_7, C_{11}\}$, let $y_d$, $y_d'$ be the two vertices of $H$ which are distance $d$ from $y$ on the cycle. Take 
\begin{displaymath}
    D_{H} = \left\{
        \begin{array}{ll}
        \{y\},& \text{if}\ H \in \{P_3, C_3\},\\
        \{y_3, y_3'\},& \text{if}\  H \cong C_7,\\
        \{y_2,y_2', y_5'\},& \text{if}\  H \cong C_{11}.\\
        \end{array} \right.
\end{displaymath}
Note that $D_X=\{v\} \cup \bigcup_{H \in \mathcal{H}_b}D_H$ is a 1-isolating set of $G[X]$, as shown in Fig. \hyperlink{Fig6}{6}.

\begin{figure}[h!]
	\begin{center}
		\begin{tikzpicture}[scale=.48]
			\tikzstyle{vertex}=[circle, draw, inner sep=0pt, minimum size=6pt]
			\tikzset{vertexStyle/.append style={rectangle}}
			\vertex (1) at (0,0) [scale=.75,fill=lightgray] {};
			\node ($x$) at (0.7,-0.5) {$x$};
			\vertex (2) at (0,-2) [scale=.75] {};
			\vertex (3) at (-1,-2.8) [scale=.75,fill=lightgray] {};
			\vertex (4) at (-1,-3.6) [scale=.75,fill=black] {};
			\vertex (5) at (-0.5,-4.4) [scale=.75,fill=lightgray] {};
			\vertex (6) at (0.5,-4.4) [scale=.75,fill=lightgray] {};
			\vertex (7) at (1,-3.6) [scale=.75,fill=black] {};
			\vertex (8) at (1,-2.8) [scale=.75,fill=lightgray] {};
			
			\path
			(1) edge (2)
			(2) edge (3)
			(3) edge (4)
			(4) edge (5)
			(5) edge (6)
			(6) edge (7)
			(7) edge (8)
			(2) edge (8)
			;
			
			\vertex (12) at (-2,0) [scale=.75] {};
			\node ($y$) at (-1.6,-0.5) {$y$};
			\vertex (13) at (-2.6,1) [scale=.75,fill=lightgray] {};
			\vertex (14) at (-3.6,1) [scale=.75,fill=black] {};
			\node ($y_2$) at (-3.6,1.6) {$y_2$};
			\vertex (15) at (-4.6,1) [scale=.75,fill=lightgray] {};
			\vertex (16) at (-5.6,1) [scale=.75] {};
			\vertex (17) at (-6.6,0.5) [scale=.75,fill=lightgray] {};
			\vertex (18) at (-6.6,-0.5) [scale=.75,fill=black] {};
			\node ($y_5'$) at (-7.2,-0.8) {$y_5'$};
			\vertex (19) at (-5.6,-1) [scale=.75,fill=lightgray] {};
			\vertex (20) at (-4.6,-1) [scale=.75,fill=lightgray] {};
			\vertex (21) at (-3.6,-1) [scale=.75,fill=black] {};
			\node ($y_2'$) at (-3.6,-1.8) {$y_2'$};
			\vertex (22) at (-2.6,-1) [scale=.75,fill=lightgray] {};
			
			\path
			(1) edge (12)
			(12) edge (13)
			(13) edge (14)
			(14) edge (15)
			(15) edge (16)
			(16) edge (17)
			(17) edge (18)
			(18) edge (19)
			(19) edge (20)
			(20) edge (21)
			(21) edge (22)
			(22) edge (12)
			;

			\vertex (9) at (-1.7,-1.7) [scale=.75,fill=black] {};
			\vertex (10) at (-2.6,-2.6) [scale=.75,fill=lightgray] {};
			\vertex (11) at (-3.5,-3.5) [scale=.75] {};
			
			\path
			(1) edge (9)
			(9) edge (10)
			(10) edge (11)
			;
			
			\vertex (23) at (-1.5,1.7) [scale=.75,fill=black] {};
			\vertex (24) at (-2.9,2) [scale=.75,fill=lightgray] {};
			\vertex (25) at (-1.8,3) [scale=.75,fill=lightgray] {};
			
			\path
			(1) edge (23)
			(23) edge (24)
			(23) edge (25)
			;
			
			\vertex (26) at (0,2) [scale=.75,fill=black] {};
			\vertex (27) at (-0.7,3) [scale=.75,fill=lightgray] {};
			\vertex (28) at (0.7,3) [scale=.75,fill=lightgray] {};
			
			\path
			(1) edge (26)
			(26) edge (27)
			(27) edge (28)
			(26) edge (28)
			;
			
			\vertex (29) at (2,0) [scale=.75,fill=black] {};
			\node ($v$) at (1.7,-0.5) {$v$};
			\vertex (30) at (4.5,3) [scale=.75,fill=lightgray] {};
			\vertex (31) at (4.5,1.5) [scale=.75,fill=lightgray] {};
			\vertex (32) at (4.5,0) [scale=.75,fill=lightgray] {};
			\vertex (33) at (4.5,-0.85) [scale=.5,fill=black] {};
			\vertex (34) at (4.5,-1.5) [scale=.5,fill=black] {};
			\vertex (35) at (4.5,-2.15) [scale=.5,fill=black] {};
			\vertex (36) at (4.5,-3) [scale=.75,fill=lightgray] {};
			
			\path
			(29) edge (30)
			(29) edge (31)
			(29) edge (32)
			(29) edge (36)
			(1)  edge (29)
			;
			
			\vertex (37) at (0.7,-1.55) [scale=.01,fill=black] {};
			\vertex (38) at (6.2,3) [scale=.01] {};
			\vertex (39) at (6.5,1.5) [scale=.05] {};
			\vertex (40) at (6.2,0) [scale=.01] {};
			\vertex (200) at (6.5,-1.45) [scale=.05] {};
			\vertex (400) at (6.4,2.4) [scale=.01,fill=black] {};

			\path
			(30) edge (38) [dashed]
			(32) edge (40)
			(36) edge (200)
			(1)  edge (37)    
			(26) edge (30)
			(1)  edge (30)  
			(6) edge (36)
			(31) edge (400)
			;
			
			\draw (8,1.5) ellipse (1.5 and 1);
			\node [scale=0.8]($H$) at (8,1.5) {$H \in \mathcal{H}_g$};
			\path
			(31) edge (39)
			;
			
			\vertex (41) at (6.5,-3) [scale=.75] {};
			\vertex (42) at (7.3,-4) [scale=.75,fill=lightgray] {};
			\vertex (43) at (8.3,-4) [scale=.75,fill=black] {};
			\vertex (44) at (9.3,-3.5) [scale=.75,fill=lightgray] {};
			\vertex (45) at (9.3,-2.5) [scale=.75,fill=lightgray] {};
			\vertex (46) at (8.3,-2) [scale=.75,fill=black] {};
			\vertex (47) at (7.3,-2) [scale=.75,fill=lightgray] {};
			
			\path
			(36) edge (41)
			(41) edge (42)
			(42) edge (43)
			(43) edge (44)
			(44) edge (45)
			(45)  edge (46)
			(46) edge (47)
			(47) edge (41)
			;
			
			\draw (0.8,3.1) parabola bend (2.3,3.3) (4.3,3) [dashed];
			
			\draw (-5.5,1.1) parabola bend (-0.3,4.6) (4.3,3.23) [dashed];   
			
			\draw (4.65,-0.06)..controls (5.5,-0.3) and (8.2,-0.6) ..(9.3,-2.35)[dashed];
			
			\draw (-1.7,-1.84)..controls (-2.6,-5) and (0,-5.4) ..(-0.2,-5.2)[dashed];
			
			\draw (-0.18,-5.25)..controls (0.6,-5.3) and (2.3,-5) ..(4.4,-3.2)[dashed];
			
			\draw (-1.5,1.84)..controls (-1.2,4.5) and (1.1,4.5) ..(4.3,3.1)[dashed];
		\end{tikzpicture}
	\end{center}
	\vspace{1.5mm}
	\par {\footnotesize \centerline{{\bf Fig. 6.} ~The set $D_X$. \hypertarget{Fig6}}}
\end{figure}
    
    By Lemmas \ref{Lemma2.2} and \ref{Lemma2.3}, and by the induction hypothesis, we have 
    $$\begin{aligned}
    \iota_1(G)& \leq |D_X|+\iota_1(G-X) = |\{v\}|+\sum_{H \in \mathcal{H}_b}|D_H| + \sum_{H \in \mathcal{H}_g} \iota_1(H)\\
              & \leq 1+k_3+2k_7+3k_{11}+
              \frac{1}{4} (n-\Delta-1-3k_3-7k_7-11k_{11}) \\
              &= \frac{n}{4}+ \frac{1}{4}(3-\Delta+k_3+k_7+k_{11}).
    \end{aligned}$$
If $\Delta \geq k_3+k_7+k_{11}+3=|\mathcal{H}_b|+3$, then $\iota_1(G) \leq \frac{n}{4}$. Hence, it remains to consider the case $\Delta \leq |\mathcal{H}_b|+2$. Take
\begin{displaymath}
    D_{H}' = \left\{
        \begin{array}{ll}
        \{x\},& \text{if}\ H \in \{P_3, C_3\},\\
        \{x, y_3\},& \text{if}\  H \cong C_7,\\
        \{x, y_3,y_3'\},& \text{if}\  H \cong C_{11}.\\
        \end{array} \right.
\end{displaymath}
Clearly, $D_X'=N(v) \cup \bigcup_{H \in \mathcal{H}_b}D_{H}'$ is a 1-isolating set of $G[X]$, as shown in Fig. \hyperlink{Fig7}{7}.

\begin{figure}[h!]
	\begin{center}
		\begin{tikzpicture}[scale=.48]
			\tikzstyle{vertex}=[circle, draw, inner sep=0pt, minimum size=6pt]
			\tikzset{vertexStyle/.append style={rectangle}}
			\vertex (1) at (0,0) [scale=.75,fill=black] {};
			\node ($x$) at (0.7,-0.5) {$x$};
			\vertex (2) at (0,-2) [scale=.75,fill=lightgray] {};
			\vertex (3) at (-1,-2.8) [scale=.75] {};
			\vertex (4) at (-1,-3.6) [scale=.75,fill=lightgray] {};
			\vertex (5) at (-0.5,-4.4) [scale=.75,fill=black] {};
			\vertex (6) at (0.5,-4.4) [scale=.75,fill=lightgray] {};
			\vertex (7) at (1,-3.6) [scale=.75] {};
			\vertex (8) at (1,-2.8) [scale=.75] {};

			\path
			(1) edge (2)
			(2) edge (3)
			(3) edge (4)
			(4) edge (5)
			(5) edge (6)
			(6) edge (7)
			(7) edge (8)
			(2) edge (8)
			;
			
			\vertex (12) at (-2,0) [scale=.75,fill=lightgray] {};
			\node ($y$) at (-1.6,-0.5) {$y$};
			\vertex (13) at (-2.6,1) [scale=.75] {};
			\vertex (14) at (-3.6,1) [scale=.75,fill=lightgray] {};
			\vertex (15) at (-4.6,1) [scale=.75,fill=black] {};
			\node ($y_3$) at (-4.2,1.6) {$y_3$};
			\vertex (16) at (-5.6,1) [scale=.75,fill=lightgray] {};
			\vertex (17) at (-6.6,0.5) [scale=.75] {};
			\vertex (18) at (-6.6,-0.5) [scale=.75] {};
			\vertex (19) at (-5.6,-1) [scale=.75,fill=lightgray] {};
			\vertex (20) at (-4.6,-1) 
			[scale=.75,fill=black] {};\node ($y_3'$) at (-4.6,-1.8) {$y_3'$};
			\vertex (21) at (-3.6,-1) [scale=.75,fill=lightgray] {};
			\vertex (22) at (-2.6,-1) [scale=.75] {};
			
			\path
			(1) edge (12)
			(12) edge (13)
			(13) edge (14)
			(14) edge (15)
			(15) edge (16)
			(16) edge (17)
			(17) edge (18)
			(18) edge (19)
			(19) edge (20)
			(20) edge (21)
			(21) edge (22)
			(22) edge (12)
			;
			
			\vertex (9) at (-1.7,-1.7) [scale=.75,fill=lightgray] {};
			\vertex (10) at (-2.6,-2.6) [scale=.75] {};
			\vertex (11) at (-3.5,-3.5) [scale=.75] {};
			
			\path
			(1) edge (9)
			(9) edge (10)
			(10) edge (11)
			;
			
			\vertex (23) at (-1.5,1.7) [scale=.75,fill=lightgray] {};
			\vertex (24) at (-2.9,2) [scale=.75] {};
			\vertex (25) at (-1.8,3) [scale=.75] {};
			
			\path
			(1) edge (23)
			(23) edge (24)
			(23) edge (25)
			;
			
			\vertex (26) at (0,2) [scale=.75,fill=lightgray] {};
			\vertex (27) at (-0.7,3) [scale=.75] {};
			\vertex (28) at (0.7,3) [scale=.75] {};
			
			\path
			(1) edge (26)
			(26) edge (27)
			(27) edge (28)
			(26) edge (28)
			;
			
			\vertex (29) at (2,0) [scale=.75,fill=lightgray] {};
			\node ($v$) at (1.7,-0.5) {$v$};
			\vertex (30) at (4.5,3) [scale=.75,fill=black] {};
			\vertex (31) at (4.5,1.5) [scale=.75,fill=black] {};
			\vertex (32) at (4.5,0) [scale=.75,fill=black] {};
			\vertex (33) at (4.5,-0.85) [scale=.5,fill=black] {};
			\vertex (34) at (4.5,-1.5) [scale=.5,fill=black] {};
			\vertex (35) at (4.5,-2.15) [scale=.5,fill=black] {};
			\vertex (36) at (4.5,-3) [scale=.75,fill=black] {};
			
			\path
			(29) edge (30)
			(29) edge (31)
			(29) edge (32)
			(29) edge (36)
			(1)  edge (29)
			;
			
			\vertex (37) at (0.7,-1.55) [scale=.01,fill=black] {};
			\vertex (38) at (6.22,3) [scale=.01] {};
			\vertex (39) at (6.5,1.5) [scale=.05] {};
			\vertex (40) at (6.22,0) [scale=.01] {};
			\vertex (200) at (6.5,-1.45) [scale=.05] {};
			\vertex (400) at (6.4,2.4) [scale=.01,fill=black] {};
			
			\path
			(30) edge (38) [dashed]
			(32) edge (40)
			(1)  edge (37)
			(26) edge (30)
			(1)  edge (30)
			(6) edge (36)
			(36) edge (200)
			(31) edge (400)
			;

			\draw (8,1.5) ellipse (1.5 and 1);
			\node [scale=0.8] ($H$) at (8,1.5) {$H \in \mathcal{H}_g$};
			\path
			(31) edge (39)
			;
			
			\vertex (41) at (6.5,-3) [scale=.75,fill=lightgray] {};
			\vertex (42) at (7.3,-4) [scale=.75] {};
			\vertex (43) at (8.3,-4) [scale=.75] {};
			\vertex (44) at (9.3,-3.5) [scale=.75,fill=lightgray] {};
			\vertex (45) at (9.3,-2.5) [scale=.75,fill=black] {};
			\vertex (46) at (8.3,-2) [scale=.75,fill=lightgray] {};
			\vertex (47) at (7.3,-2) [scale=.75] {};
			
			\path
			(36) edge (41)
			(41) edge (42)
			(42) edge (43)
			(43) edge (44)
			(44) edge (45)
			(45)  edge (46)
			(46) edge (47)
			(47) edge (41)
			;
			
			\draw (0.8,3.1) parabola bend (2.3,3.3) (4.3,3) [dashed];
			
			\draw (-5.5,1.1) parabola bend (-0.3,4.6) (4.3,3.23) [dashed];
			
			\draw (4.65,-0.06)..controls (5.5,-0.3) and (8.2,-0.6) ..(9.3,-2.35)[dashed];
			
			\draw (-1.7,-1.84)..controls (-2.6,-5) and (0,-5.4) ..(-0.2,-5.2)[dashed];
			
			\draw (-0.18,-5.25)..controls (0.6,-5.3) and (2.3,-5) ..(4.4,-3.2)[dashed];
			
			\draw (-1.5,1.84)..controls (-1.2,4.5) and (1.1,4.5) ..(4.3,3.1)[dashed];
		\end{tikzpicture}
	\end{center}
	\vspace{1.5mm}
	\par {\footnotesize \centerline{{\bf Fig. 7.} ~The set $D_X'$.\hypertarget{Fig7}}}
\end{figure}

By Lemmas \ref{Lemma2.2} and \ref{Lemma2.3}, and by the induction hypothesis, we have 
$$\begin{aligned}
    \iota_1(G)& \leq |D_X'|+\iota_1(G-X) \leq |N(v)| + \sum_{H \in \mathcal{H}_b} |D_{H}' \setminus \{x\}|
              + \sum_{H \in \mathcal{H}_g} \iota_1(H)\\
              & \leq \Delta + k_7+2k_{11}+
              \frac{1}{4} (n-\Delta-1-3k_3-7k_7-11k_{11}) \\
              &= \frac{n}{4}+ \frac{1}{4}(3\Delta-1-3k_3-3k_7-3k_{11}).
\end{aligned}$$
If $\Delta \leq k_3+k_7+k_{11}=|\mathcal{H}_b|$, then $\iota_1(G) \leq \frac{n}{4}$. Hence, we may assume that $|\mathcal{H}_b|+1 \leq \Delta \leq |\mathcal{H}_b|+2$. This proves Claim \ref{ClaimC}.
\end{proof}

In terms of the value of $\Delta$, we distinguish the remaining proof into three cases.

\vspace{3mm}

{\bf Case 1.} $\Delta \geq 5$.

\vspace{3mm}

Let $X, D_H, D_H'$ and $D_X$ be the sets defined as in the proof of Claim \ref{ClaimC}. %Now we construct a new 1-isolating set of $G[X]$ of smaller order based on $D_X$.

Since $|\mathcal{H}_b| \geq \Delta-2$, and for each $H \in \mathcal{H}_b$, $|N(H)|\geq 2$, we have $\sum_{H \in \mathcal{H}_b}|N(H)|\geq 2|\mathcal{H}_b|\geq 2(\Delta -2)\geq \Delta+1 = |N(v)| + 1$.  
By the Pigeonhole Principle, there exists a vertex $x \in N(v)$ such that $x \in N(H_1) \cap N(H_2)$ for some $H_1, H_2 \in \mathcal{H}_b$. Then, $D_X'' = D_{H_1}' \cup D_{H_2}' \cup \bigcup_{H \in \mathcal{H}_b \setminus \{H_1,H_2\}} D_H \cup \{v\}= [D_X \setminus (D_{H_1} \cup D_{H_2})] \cup (D_{H_1}' \cup D_{H_2}')$ is a 1-isolating set of $G[X]$. Clearly, $|D_X''|=|D_X|-1$. One can see Fig. \hyperlink{Fig8}{8} for an example from $D_X$ to $D_X''$ where $H_1 \cong C_3$ and $H_2 \cong C_{11}$.

    \begin{figure}[h!]
	\begin{center}
		\begin{tikzpicture}[scale=.48]
			\tikzstyle{vertex}=[circle, draw, inner sep=0pt, minimum size=6pt]
			\tikzset{vertexStyle/.append style={rectangle}}
			\vertex (1) at (-1,0) [scale=.75,fill=lightgray] {};
			\node ($x$) at (-1.5,-0.5) {$x$};
			\vertex (2) at (-1,-2) [scale=.75] {};
			\vertex (3) at (-2,-2.8) [scale=.75,fill=lightgray] {};
			\vertex (4) at (-2,-3.6) [scale=.75,fill=black] {};
			\vertex (5) at (-2,-4.4) [scale=.75,fill=lightgray] {};
			\vertex (6) at (0,-4.4) [scale=.75,fill=lightgray] {};
			\vertex (7) at (0,-3.6) [scale=.75,fill=black] {};
			\vertex (8) at (0,-2.8) [scale=.75,fill=lightgray] {};
			\vertex (52) at (-2,-5.2) [scale=.75,fill=lightgray] {};
			\vertex (55) at (0,-5.2) [scale=.75] {};
			\vertex (53) at (-1.5,-6) [scale=.75,fill=black] {};
			\vertex (54) at (-0.5,-6) [scale=.75,fill=lightgray] {};
			
			\path
			(1) edge (2)
			(2) edge (3)
			(3) edge (4)
			(4) edge (5)
			(5) edge (52)
			(52) edge (53)
			(53) edge (54)
			(54) edge (55)
			(55) edge (6)
			(6) edge (7)
			(7) edge (8)
			(2) edge (8)
			;
			
			\vertex (26) at (-1,2) [scale=.75,fill=black] {};
			\vertex (27) at (-1.7,3) [scale=.75,fill=lightgray] {};
			\vertex (28) at (-0.3,3) [scale=.75,fill=lightgray] {};
			
			\path
			(1) edge (26)
			(26) edge (27)
			(27) edge (28)
			(26) edge (28)
			;
			
			\vertex (29) at (2,0) [scale=.75,fill=black] {};
			\node ($v$) at (1.7,-0.5) {$v$};
			\vertex (30) at (4.5,3) [scale=.75,fill=lightgray] {};
			\vertex (31) at (4.5,1.5) [scale=.75,fill=lightgray] {};
			\vertex (32) at (4.5,0) [scale=.75,fill=lightgray] {};
			\vertex (33) at (4.5,-1.1) [scale=.5,fill=black] {};
			\vertex (34) at (4.5,-1.8) [scale=.5,fill=black] {};
			\vertex (35) at (4.5,-2.5) [scale=.5,fill=black] {};
			\vertex (36) at (4.5,-3.6) [scale=.75,fill=lightgray] {};
			
			\path
			(29) edge (30)
			(29) edge (31)
			(29) edge (32)
			(29) edge (36)
			(1)  edge (29)
			;

			\vertex (37) at (-2.7,0) [scale=.05] {};
			\vertex (38) at (6.2,3) [scale=.01] {};
			\vertex (39) at (6.5,1.5) [scale=.05] {};
			\vertex (40) at (6.2,0) [scale=.01] {};
			\vertex (300) at (6.15,-2.2) [scale=.05] {};
			\vertex (400) at (6.3,2.4) [scale=.01,fill=black] {};
			
			\path
			(30) edge (38) [dashed]
			(32) edge (40)
			(1)  edge (37)
			(28) edge (30)
			(1)  edge (30)
			(7) edge (36)
			(2) edge (36)
			(36) edge (300)
			(31) edge (400)
			;
			
			\draw (8,1.5) ellipse (1.5 and 1);
			\node [scale=0.8]($H$) at (8,1.5) {$H \in \mathcal{H}_g$};
			\path
			(31) edge (39)
			;
			
			\vertex (41) at (6.5,-3.6) [scale=.75] {};
			\vertex (42) at (7.3,-4.6) [scale=.75,fill=lightgray] {};
			\vertex (43) at (8.3,-4.6) [scale=.75,fill=black] {};
			\vertex (44) at (9.3,-4.1) [scale=.75,fill=lightgray] {};
			\vertex (45) at (9.3,-3.1) [scale=.75,fill=lightgray] {};
			\vertex (46) at (8.3,-2.6) [scale=.75,fill=black] {};
			\vertex (47) at (7.3,-2.6) [scale=.75,fill=lightgray] {};
			
			\path
			(36) edge (41)
			(41) edge (42)
			(42) edge (43)
			(43) edge (44)
			(44) edge (45)
			(45)  edge (46)
			(46) edge (47)
			(47) edge (41)
			;
			
			%rectangle
			\vertex (49) at (-3,3.8) [scale=.05] {};
			\vertex (50) at (1,3.8) [scale=.05] {};
			\vertex (51) at (1,-6.8) [scale=.05] {};
			\vertex (52) at (-3,-6.8) [scale=.05] {};
			
			\path
			(49) edge (50) [dashed]
			(50) edge (51)
			(51) edge (52)
			(52) edge (49)
			;
			
			%the arrow 
			\vertex (200) at (9.7,-1.1) [scale=.04,fill=black] {};
			\vertex (201) at (11.7,-1.1) [scale=.012] {};
			\draw[thick][->,>=latex](200)--(201);
			
			%the right graph    
			\vertex (61) at (14.5,0) [scale=.75,fill=black] {};
			\node ($x$) at (14,-0.5) {$x$};
			\vertex (62) at (14.5,-2) [scale=.75,fill=lightgray] {};
			\vertex (63) at (13.5,-2.8) [scale=.75] {};
			\vertex (64) at (13.5,-3.6) [scale=.75,fill=lightgray] {};
			\vertex (65) at (13.5,-4.4) [scale=.75,fill=black] {};
			\vertex (66) at (15.5,-4.4) [scale=.75,fill=black] {};
			\vertex (67) at (15.5,-3.6) [scale=.75,fill=lightgray] {};
			\vertex (68) at (15.5,-2.8) [scale=.75] {};
			\vertex (112) at (13.5,-5.2) [scale=.75,fill=lightgray] {};
			\vertex (115) at (15.5,-5.2) [scale=.75,fill=lightgray] {};
			\vertex (113) at (14,-6) [scale=.75] {};
			\vertex (114) at (15,-6) [scale=.75] {};

			\path
			(61) edge (62)
			(62) edge (63)
			(63) edge (64)
			(64) edge (65)
			(65) edge (112)
			(112) edge (113)
			(113) edge (114)
			(114) edge (115)
			(115) edge (66)
			(66) edge (67)
			(67) edge (68)
			(62) edge (68)
			;
			
			\vertex (86) at (14.5,2) [scale=.75,fill=lightgray] {};
			\vertex (87) at (13.8,3) [scale=.75] {};
			\vertex (88) at (15.2,3) [scale=.75] {};
			
			\path
			(61) edge (86)
			(86) edge (87)
			(87) edge (88)
			(86) edge (88)
			;
			
			\vertex (89) at (17.5,0) [scale=.75,fill=black] {};
			\node ($v$) at (17.2,-0.5) {$v$};
			\vertex (90) at (20,3) [scale=.75,fill=lightgray] {};
			\vertex (91) at (20,1.5) [scale=.75,fill=lightgray] {};
			\vertex (92) at (20,0) [scale=.75,fill=lightgray] {};
			\vertex (93) at (20,-1.1) [scale=.5,fill=black] {};
			\vertex (94) at (20,-1.8) [scale=.5,fill=black] {};
			\vertex (95) at (20,-2.5) [scale=.5,fill=black] {};
			\vertex (96) at (20,-3.6) [scale=.75,fill=lightgray] {};
			
			\path
			(89) edge (90)
			(89) edge (91)
			(89) edge (92)
			(89) edge (96)
			(61)  edge (89)
			;
			
			\vertex (97) at (12.8,0) [scale=.05] {};
			\vertex (98) at (21.7,3) [scale=.02] {};
			\vertex (99) at (22,1.5) [scale=.05] {};
			\vertex (100) at (21.7,0) [scale=.02] {};
			\vertex (301) at (21.65,-2.2) [scale=.05] {};
			\vertex (402) at (21.8,2.4) [scale=.05,fill=black] {};
			
			\path
			(90) edge (98) [dashed]
			(92) edge (100)
			(61)  edge (97)
			(88) edge (90)
			(61)  edge (90)
			(67) edge (96)
			(62) edge (96)
			(96) edge (301)
			(91) edge (402)
			;
			
			\draw (23.5,1.5) ellipse (1.5 and 1);
			\node [scale=0.8] ($H$) at (23.5,1.5) {$H \in \mathcal{H}_g$};
			\path
			(91) edge (99)
			;
			
			\vertex (101) at (22,-3.6) [scale=.75] {};
			\vertex (102) at (22.8,-4.6) [scale=.75,fill=lightgray] {};
			\vertex (103) at (23.8,-4.6) [scale=.75,fill=black] {};
			\vertex (104) at (24.8,-4.1) [scale=.75,fill=lightgray] {};
			\vertex (105) at (24.8,-3.1) [scale=.75,fill=lightgray] {};
			\vertex (106) at (23.8,-2.6) [scale=.75,fill=black] {};
			\vertex (107) at (22.8,-2.6) [scale=.75,fill=lightgray] {};
			
			\path
			(96) edge (101)
			(101) edge (102)
			(102) edge (103)
			(103) edge (104)
			(104) edge (105)
			(105)  edge (106)
			(106) edge (107)
			(107) edge (101)
			;
			
			%rectangle
			\vertex (109) at (12.5,3.8) [scale=.05] {};
			\vertex (110) at (16.5,3.8) [scale=.05] {};
			\vertex (111) at (16.5,-6.8) [scale=.05] {};
			\vertex (112) at (12.5,-6.8) [scale=.05] {};
			
			\path
			(109) edge (110) [dashed]
			(110) edge (111)
			(111) edge (112)
			(112) edge (109)
			;   
			
			\draw (4.65,-0.06)..controls (5.5,-0.5) and (8.2,-1.1) ..(9.3,-2.95)[dashed];
			\draw (20.15,-0.06)..controls (21,-0.5) and (23.7,-1.1) ..(24.8,-2.95)[dashed];
			
		\end{tikzpicture}
	\end{center}
	\vspace{1.5mm}
	\par {\footnotesize \centerline{{\bf Fig. 8.} ~The sets $D_X$ and $D_X''$ where $H_1 \cong C_3$ and $H_2 \cong C_{11}$.\hypertarget{Fig8}}}
\end{figure}

    Recall that $|\mathcal{H}_b| \leq \Delta -1<\Delta +1$. By Lemmas \ref{Lemma2.2} and \ref{Lemma2.3}, and by the induction hypothesis, we have
    $$\begin{aligned}
        \iota_1(G)& \leq |D_X''|+\iota_1(G-X) = |D_X|-1+ \sum_{H \in \mathcal{H}_g} \iota_1(H)\\
                  &\leq  \frac{n}{4}+ \frac{1}{4}(k_3+k_7+k_{11}-\Delta +3)-1=  \frac{n}{4}+ \frac{1}{4}(|\mathcal{H}_b|- \Delta -1) < \frac{n}{4}.             
    \end{aligned}$$

\vspace{3mm}

Now we fix $x\in N(v)$ with the property that there exists some $H^* \in \mathcal{H}_b$ with $x \in N(H^*)$. Let $\mathcal{H}_g^x$ be the set of components $H$ of $\mathcal{H}_g$ with $N(H) = \{x\}$. By the induction hypothesis, $\iota_1(H) \leq \frac{1}{4}|V(H)|$ for any component $H \in \mathcal{H}_g^x \subseteq \mathcal{H}_g$.

\vspace{3mm}

{\bf Case 2.} $\Delta=3$.

\vspace{3mm}

Let $X=V(H^*)  \cup  \{x\}$. Then $G-X= G_v \cup \bigcup_{H \in \mathcal{H}_g^x}H$, where $G_v$ is the component of $G-X$ containing $v$. Let $Y=X \cup V(G_v)$. Then $G-Y=\bigcup_{H \in \mathcal{H}_g^x}H$. Let $v_{d}$, $v_{d}'$ be the vertices distance $d$ from $v$ in $G_v$. Let $xy \in E(G)$ for some $y \in V(H^*)$, and let $y_{d}$, $y_{d}'$ be the vertices distance $d$ from $y$ in $H^*$ if $H^* \in \{C_3,C_7,C_{11}\}$.

\emph{Subcase 2.1.} $H^* \cong C_3$. Clearly, $E(\{y_1,y_1'\},\{v_1,v_1'\}) \neq \emptyset$, by $|N(H^*)| \geq 2$ and $d(y) = \Delta = 3$. Assume that $y_1v_1 \in E(G)$. We now consider the structure of $G_v$.

\emph{Subcase 2.1.1.} $G_v \notin \mathcal{S}$. Clearly, $\{y\}$ is a 1-isolating set of $G[X]$, and $E(X \setminus N[y], V \setminus X) = \emptyset$. By Lemmas \ref{Lemma2.2} and \ref{Lemma2.3}, and by the induction hypothesis, $\iota_1(G) \leq |\{y\}|+ \iota_1(G-X) = 1 + \iota_1(G_v) + \sum_{H \in \mathcal{H}_g^x}\iota_1(H) \leq 1+ \frac{1}{4}(n-4) = \frac{n}{4}$. 

\emph{Subcase 2.1.2.} $G_v \cong P_3$. Note that $y_1yxvv_1y_1$ is a 5-cycle in $G$. Since $G$ contains no induced 5-cycles, $xv_1 \in E(G)$ and $d(x)=\Delta=3$. Thus, $\mathcal{H}_g^x = \emptyset$ and $G=G[Y]$. If $y_1'v_1' \in E(G)$, then $y_1'v_1'vxyy_1'$ is an induced 5-cycle. Let $y_1'v_1' \notin E(G)$. Then, $\{x\}$ is a 1-isolating set of $G$, and $\iota_1(G) \leq |\{x\}|=1 < \frac{7}{4} = \frac{n}{4}$. 

\emph{Subcase 2.1.3.} $G_v \in \{C_3,C_7,C_{11}\}$. Since $y_1v_1 \in E(G)$ and $d(y_1)=d(y)=d(v)=d(v_1)=\Delta=3$, $yy_1v_1vxy$ is an induced 5-cycle in $G$, a contradiction. 

\emph{Subcase 2.2.} $H^* \cong P_3$. We consider the degree of $y$ in $H^*$.

\emph{Subcase 2.2.1.} $d_{H^*}(y)$=2. Let $N_{H^*}(y) = \{y_1,y_1'\}$. Since $|N(H^*)| \geq 2$ and $d(y)=\Delta=3$, $E(\{y_1,y_1'\},\{v_1,v_1'\}) \neq \emptyset$. Assume that $y_1v_1 \in E(G)$.  We further consider the structure of $G_v$. 

(i) $G_v \notin \mathcal{S}$. Clearly, $\{y\}$ is a 1-isolating set of $G[X]$, and $E(X \setminus N[y], V \setminus X) = \emptyset$. By Lemmas \ref{Lemma2.2} and \ref{Lemma2.3}, and by the induction hypothesis, $\iota_1(G) \leq |\{y\}|+ \iota_1(G-X) = 1 + \iota_1(G_v) + \sum_{H \in \mathcal{H}_g^x}\iota_1(H) \leq 1+ \frac{1}{4}(n-4) = \frac{n}{4}$. 

%By the non-existence of induced 5-cycles 

(ii) $G_v \cong P_3$. Note that $y_1yxvv_1y_1$ can not be an induced 5-cycle in $G$. Clearly, $y_1x \in E(G)$ or $v_1x \in E(G)$, and $d(x)=\Delta=3$. Thus, $\mathcal{H}_g^x = \emptyset$ and $G=G[Y]$. Since $G$ contains no induced 5-cycles, $G-N[x]$ consists of three isolated vertices. This implies that, $\{x\}$ is a 1-isolating set of $G$, and $\iota_1(G) \leq |\{x\}|=1 < \frac{7}{4} = \frac{n}{4}$.

(ii) $G_v \in \{C_3, C_7, C_{11}\}$. Since $d(v_1)=\Delta=3$ and $y_1yxvv_1y_1$ is not an induced 5-cycle in $G$, $y_1x \in E(G)$ and $d(x)=\Delta=3$. Thus, $\mathcal{H}_g^x = \emptyset$ and $G=G[Y]$. Take 
\begin{displaymath}
    D = \left\{
        \begin{array}{ll}
        \{x\},& \text{if}\ G_v \cong C_3,\\
        \{x, v_3\},& \text{if}\  G_v \cong C_7,\\
        \{x, v_3, v_3'\},& \text{if}\  G_v \cong C_{11}.\\
        \end{array} \right.
\end{displaymath}
If $y_1'v_1' \in E(G)$, then $yy_1'v_1'vxy$ is an induced 5-cycle in $G$, a contradiction. Hence, $D$ is a 1-isolating set of $G$, and $\iota_1(G) \leq |D|= \frac{1}{4}(n-3) < \frac{n}{4}$.

\emph{Subcase 2.2.2.} $d_{H^*}(y)$=1. Let $N_{H^*}(y_1)=\{y,y_2\}$. Note that $d_{H^*}(y_1)=2$ and $|N(H^*)| \geq 2$. If $y_1v' \in E(G)$ for some $v' \in N(v)$, then regarding $v'$ and $y_1$ separately as $x$ and $y$, this subcase can come down to Subcase 2.2.1. So, let $N(y_1) \cap N(v) = \emptyset$. We further consider whether $N(y_2) \cap \{v_1,v_1'\}= \emptyset$ or not.

(i) $N(y_2) \cap \{v_1,v_1'\}= \emptyset$. Note that $\{y\}$ is a 1-isolating set of $G[X]$. If $G_v$ is not an $\mathcal{S}$-graph, then by Lemmas \ref{Lemma2.2} and \ref{Lemma2.3}, and by the induction hypothesis, $\iota_1(G) \leq |\{y\}|+ \iota_1(G-X) \leq 1+\frac{1}{4}(n-4)=\frac{n}{4}$. If $G_v$ is an $\mathcal{S}$-graph, then take
    \begin{displaymath}
        D = \left\{
        \begin{array}{ll}
        \{x\},& \text{if}\ G_v\in \{P_3, C_3\},\\
        \{x, v_3\},& \text{if}\ G_v\cong C_{7},\\
        \{x, v_3,v_3'\},& \text{if}\ G_v\cong C_{11}.\\
        \end{array} \right.
    \end{displaymath}
    Clearly, $D$ is a 1-isolating set of $G[Y]$. By Lemma \ref{Lemma2.2} and the induction hypothesis,  $\iota_1(G) \leq |D|+ \iota_1(G-Y) = \frac{1}{4}(|Y|-3) + \sum_{H \in \mathcal{H}_g^x}\iota_1(H) \leq \frac{1}{4}(|Y|-3) + \frac{1}{4}(n-|Y|) < \frac{n}{4}$. 

    (ii) $N(y_2) \cap \{v_1,v_1'\} \neq \emptyset$. Without loss of generality, we may assume that $y_2v_1 \in E(G)$. If $xv_1 \in E(G)$, then $xv_1y_2y_1yx$ is an induced 5-cycle of $G$. Let $xv_1 \notin E(G)$. Since $yy_1y_2v_1vxy$ is not an induced 6-cycle in $G$, $\emptyset \neq \{xy_2,v_1y\} \subset E(G)$.
    
%    $xv_1 \in E(G)$ or $xy_2 \in E(G)$ or $v_1y \in E(G)$.

    If $xy_2 \in E(G)$, then $d(x)=\Delta=3$. Thus, $\mathcal{H}_g^x = \emptyset$ and $G=G[Y]$. Clearly, $\{x\}$ is a 1-isolating set of $G[X]$ and $E(X \setminus N[x],V \setminus X)=\emptyset$. If $G_v$ is not an $\mathcal{S}$-graph, then by Lemma \ref{Lemma2.2} and the induction hypothesis, $\iota_1(G) \leq |\{x\}|+ \iota_1(G-X) \leq 1+\frac{1}{4}(n-4)=\frac{n}{4}$. Let $D$ be the set defined as in (i). If $G_v$ is an $\mathcal{S}$-graph, then $D$ is a 1-isolating set of $G$. Hence, $\iota_1(G) \leq |D| = \frac{1}{4}(n-3) < \frac{n}{4}$.
    
    Let $xy_2 \notin E(G)$. Then $v_1y \in E(G)$. Recall that $v_1y_2 \in E(G)$. Regarding $v_1$ as $x$, this subcase can come down to the subcase of $xy_2 \in E(G)$ above.

    %If $y_2x \notin E(G)$, then $yv_1 \in E(G)$ and $d(v_1)=d(y)=3$. Since $G$ contains no induced 5 or 6 cycles, $y_2v_1 \notin E(G)$. So, $\{y\}$ is a 1-isolating set of $G[X]$. 
    %If $G_v$ is not an $\mathcal{S}$-graph, then by the induction hypothesis, $\iota_1(G) \leq |\{y\}|+ \iota_1(G-X) \leq 1+\frac{1}{4}(n-4)=\frac{n}{4}$. If $G_v$ is an $\mathcal{S}$-graph, then since $ N(v_1) = \{y,y_2,v\}$, $G_v \cong P_3$. Hence, $\{y\}$ is a 1-isolating set of $G[Y]$. By the induction hypothesis, $\iota_1(G) \leq |\{y\}|+ \iota_1(G-Y) \leq 1+\frac{1}{4}(n-7) <\frac{n}{4}$.

\emph{Subcase 2.3.} $H^* \cong C_7$. Since $G$ contains no induced 5- and 6-cycles, we determine $E(\{y_2,y_2'\},\{v_1,v_1'\}) = \emptyset$.

\emph{Subcase 2.3.1.} $E(\{y_1,y_1'\},\{v_1,v_1'\}) \neq \emptyset$. Assume that $y_1v_1 \in E(G)$. Then, $y_1yxvv_1y_1$ is a 5-cycle in $G$, implying that $xv_1 \in E(G)$. Since $d(x)=\Delta=3$, $\mathcal{H}_g^x = \emptyset$ and $G=G[Y]$. If $y_1'v_1' \in E(G)$, then $yy_1'v_1'vxy$ is an induced 5-cycle in $G$. Hence, $\{x,y_3\}$ is a 1-isolating set of $G[X]$. If $G_v \notin \mathcal{S}$, then by Lemma \ref{Lemma2.2} and the induction hypothesis, $\iota_1(G) \leq |\{x,y_3\}|+\iota_1(G-X) \leq 2+\frac{1}{4}(n-8)=\frac{n}{4}$. If $G_v \in \mathcal{S}$, then since $d(v_1)=\Delta=3$, $G_v \cong P_3$. Clearly, $\{x,y_3\}$ is a 1-isolating set of $G$, and $\iota_1(G) \leq |\{x,y_3\}|= 2 < \frac{11}{4} = \frac{n}{4}$.

\emph{Subcase 2.3.2.} $E(\{y_1,y_1'\},\{v_1,v_1'\}) = \emptyset$. If $G_v \notin \mathcal{S}$, then $\{x,y_3\}$ is a 1-isolating set of $G[X]$. By Lemmas \ref{Lemma2.2} and \ref{Lemma2.3}, and by the induction hypothesis, $\iota_1(G) \leq |\{x,y_3\}|+\iota_1(G-X) \leq 2+\frac{1}{4}(n-8)=\frac{n}{4}$. If $G_v \in \mathcal{S}$, then take 
\begin{displaymath}
    D = \left\{
    \begin{array}{ll}
    \{x, y_3\},& \text{if}\ G_v\in \{P_3, C_3\},\\
    \{x, y_3, v_3\},& \text{if}\ G_v\cong C_{7},\\
    \{x, y_3, v_3,v_3'\},& \text{if}\ G_v\cong C_{11}.\\
    \end{array} \right.
\end{displaymath}
Clearly, $D$ is a 1-isolating set of $G[Y]$. By Lemma \ref{Lemma2.2} and the induction hypothesis, $\iota_1(G) \leq |D|+ \iota_1(G-Y) = |D|+ \sum_{H \in \mathcal{H}_g^x} \iota_1(H) \leq \frac{1}{4}(|Y|-3) + \frac{1}{4}(n-|Y|) <\frac{n}{4}$.

\emph{Subcase 2.4.} $H^* \cong C_{11}$. Since $G$ contains no induced 5- and 6-cycles, we determine $E(\{y_2,y_2'\},\{v_1,v_1'\}) = \emptyset$. 

\emph{Subcase 2.4.1.} $E(\{y_1,y_1'\},\{v_1,v_1'\}) \neq \emptyset$. Assume that $y_1v_1 \in E(G)$. Since $G$ contains no induced 5-cycles, $v_1x \in E(G)$ and $y_1'v_1' \notin E(G)$. Thus, $d(x)=d(v_1)=\Delta=3$, and $\mathcal{H}_g^x = \emptyset$ and $G=G[Y]$. Clearly, $\{y,y_4,y_4'\}$ is a 1-isolating set of $G[X]$. If $G_v \notin \mathcal{S}$, then by Lemma \ref{Lemma2.2} and the induction hypothesis, $\iota_1(G) \leq |\{y,y_4,y_4'\}|+\iota_1(G-X) \leq 3+\frac{1}{4}(n-12)=\frac{n}{4}$. If $G_v \in \mathcal{S}$, then $G_v \cong P_3$. Note that $\{x,y_4,y_4'\}$ is a 1-isolating set of $G$. Hence, $\iota_1(G) \leq |\{x,y_4,y_4'\}| = 3 < \frac{15}{4} = \frac{n}{4}$.

\emph{Subcase 2.4.2.} $E(\{y_1,y_1'\},\{v_1,v_1'\}) = \emptyset$. If $G_v \notin \mathcal{S}$, then $\{y,y_4,y_4'\}$ is a 1-isolating set of $G[X]$. By Lemmas \ref{Lemma2.2} and \ref{Lemma2.3}, and by the induction hypothesis, $\iota_1(G) \leq |\{y,y_4,y_4'\}|+\iota_1(G-X) \leq 3+\frac{1}{4}(n-12)=\frac{n}{4}$. If $G_v \in \mathcal{S}$, then take 
\begin{displaymath}
    D = \left\{
    \begin{array}{ll}
    \{x, y_4,y_4'\},& \text{if}\ G_v\in \{P_3, C_3\},\\
    \{x, y_4,y_4', v_3\},& \text{if}\ G_v\cong C_{7},\\
    \{x, y_4,y_4', v_3,v_3'\},& \text{if}\ G_v\cong C_{11}.\\
    \end{array} \right.
\end{displaymath}
Clearly, $D$ is a 1-isolating set of $G[Y]$. By Lemma \ref{Lemma2.2} and the induction hypothesis, $\iota_1(G) \leq |D|+ \iota_1(G-Y) = \frac{1}{4}(|Y|-3) +  \sum_{H \in \mathcal{H}_g^x}\iota_1(H)   \leq \frac{1}{4}(|Y|-3) + \frac{1}{4}(n-|Y|) <\frac{n}{4}$.

\vspace{3mm}

{\bf Case 3.} $\Delta=4$

\vspace{3mm}

Let $X=V(H^*) \cup \{x\}$. Then $G-X= G_v \cup \bigcup_{H \in \mathcal{H}_g^x}H$, where $G_v$ is the component of $G-X$ containing $v$. Since $d(v)=\Delta=4$, $d_{G_v}(v)=|N(v) \setminus \{x\}|=3$. It follows that $G_v$ is not an $\mathcal{S}$-graph. Set $N(v) \setminus \{x\} = \{x_1,x_2,x_3\}$.  Let $xy \in E(G)$ for some $y \in V(H^*)$, and let $y_{d}$, $y_{d}'$ be the two vertices distance $d$ from $y$ in $H^*$ if $H^* \in \{C_3,C_7,C_{11}\}$. We distinguish the following proof into three subcases in terms of the structure of $H^*$.

\emph{Subcase 3.1.} $H^* \in \{P_3,C_3\}$. We further consider the degree of $y$ in $H^*$.

\emph{Subcase 3.1.1.} $d_{H^*}(y)=2$. Clearly, $\{y\}$ is a 1-isolating set of $G[X]$, and $E(X \setminus N[y], V \setminus X) = \emptyset$. By Lemmas \ref{Lemma2.2} and \ref{Lemma2.3}, and by the induction hypothesis, $\iota_1(G) \leq |\{y\}|+\iota_1(G-X) \leq 1+\frac{1}{4}(n-4)=\frac{n}{4}$. 

\emph{Subcase 3.1.2.} $d_{H^*}(y)=1$. It follows that $H^* \cong P_3$. Let $N_{H^*}(y_1) = \{y,y_2\}$. Note that $d_{H^*}(y_1)=2$ and $|N(H^*)| \geq 2$. If $y_1x' \in E(G)$ for some $x' \in N(v)$, then regarding $x'$ and $y_1$ separately as $x$ and $y$, this subcase can come down to Subcase 3.1.1. So, let $N(y_1) \cap N(v) = \emptyset$. We now consider whether $y_2$ is adjacent to some vertices in $\{x_1,x_2,x_3\}$ or not.

Assume that $N(y_2) \cap \{x_1,x_2,x_3\} = \emptyset$. Clearly, $\{y\}$ is a 1-isolating set of $G[X]$, and $E(X \setminus N[y], V \setminus X) = \emptyset$. By Lemmas \ref{Lemma2.2} and \ref{Lemma2.3}, and by the induction hypothesis, $\iota_1(G) \leq |\{y\}| + \iota_1(G-X) \leq 1+ \frac{1}{4}(n-4) =\frac{n}{4}$. 

Assume that $N(y_2) \cap \{x_1,x_2,x_3\} \neq \emptyset$ and  $y_2x_1 \in E(G)$. Clearly, $yy_1y_2x_1vxy$ is a 6-cycle in $G$. Since $G$ contains no induced 5- and 6-cycles, $xy_2 \in E(G)$ or $x_1y \in E(G)$. If $xy_2 \in E(G)$, then $\{x\}$ is a 1-isolating set of $G[X]$. By Lemmas \ref{Lemma2.2} and \ref{Lemma2.3}, and by the induction hypothesis, $\iota_1(G) \leq |\{x\}|+\iota_1(G-X) \leq 1+\frac{1}{4}(n-4)=\frac{n}{4}$. If $xy_2 \notin E(G)$, then $x_1y \in E(G)$. Note that $x_1y_2 \in E(G)$. Regarding $x_1$ as $x$, this subcase can come down to the subcase of $xy_2 \in E(G)$ above.

\emph{Subcase 3.2.} $H^* \cong C_7$. If $N(y_2) \cap \{x_1,x_2,x_3\} = \emptyset$, then $\{y,y_3'\}$ is a 1-isolating set of $G[X]$. By Lemmas \ref{Lemma2.2} and \ref{Lemma2.3}, and by the induction hypothesis, $\iota_1(G) \leq |\{y,y_3'\}|+ \iota_1(G-X) \leq 2+\frac{1}{4}(n-8)=\frac{n}{4}$. Assume that $N(y_2) \cap \{x_1,x_2,x_3\} \neq \emptyset$ and $y_2x_1 \in E(G)$. Then $x_1y_2y_1yxvx_1$ is a 6-cycle in $G$. Since $G$ contains no induced 5- and 6-cycles, we derive that $x_1y \in E(G)$, or $xy_2 \in E(G)$, or $y_1x \in E(G)$ and $y_1x_1 \in E(G)$, or $xy_1 \in E(G)$ and $xx_1 \in E(G)$, or $x_1y_1 \in E(G)$ and $x_1x \in E(G)$.

\emph{Subcase 3.2.1.} $x_1y \in E(G)$. Assume that $N(y_2') \cap \{x_2,x_3\} = \emptyset$. Clearly, $D=\{y,y_3\}$ is a 1-isolating set of $G[X]$. Particularly, since $x_1 \in N[D]$, $x_1y_2' \in E(G)$ does not matter here. By Lemmas \ref{Lemma2.2} and \ref{Lemma2.3}, and by the induction hypothesis, $\iota_1(G) \leq |D|+ \iota_1(G-X) \leq 2+ \frac{1}{4}(n-8) = \frac{n}{4}$.

Assume that $N(y_2') \cap \{x_2,x_3\} \neq \emptyset$. Without loss of generality, let $y_2'x_2 \in E(G)$. Clearly, $yy_1'y_2'x_2vx_1y$ is a 6-cycle in $G$. Since $G$ does not contain induced 5- and 6-cycles, we derive that $x_1y_2' \in E(G)$, or $y_1'x_1 \in E(G)$ and $y_1'x_2 \in E(G)$, or $x_2y_1' \in E(G)$ and $x_2x_1 \in E(G)$. For any subcase, $d(x_1)=\Delta=4$. If $x_1y_2' \in E(G)$, then $x_1y_2y_3y_3'y_2'x_1$ is an induced 5-cycle in $G$. If $y_1'x_1 \in E(G)$, then $x_1y_2y_3y_3'y_2'y_1'x_1$ is an induced 6-cycle in $G$. Hence, $x_2y_1' \in E(G)$ and $x_2x_1 \in E(G)$. However, now $d(x_1)=d(x_2)=\Delta=4$, and $x_1y_2y_3y_3'y_2'x_2x_1$ is an induced 6-cycle in $G$.

\emph{Subcase 3.2.2.} $xy_2 \in E(G)$. We relabel the vertices of $G$ as follows: $y_2=y, y=y_2, y_3=y_1', y_3'=y_2', y_2'=y_3', y_1'=y_3, x=x_1$, and $x_1=x$. See Fig. \hyperlink{Fig9}{9} for an illustration of this procedure. Thus, this subcase comes down to Subcase 3.2.1.

\begin{figure}[h!]
    \begin{center}
    \begin{tikzpicture}[scale=.48]
    \tikzstyle{vertex}=[circle, draw, inner sep=0pt, minimum size=6pt]
    \tikzset{vertexStyle/.append style={rectangle}}
        \vertex (1) at (-0,0) [scale=.75] {};
        \node ($x$) [scale=0.9] at (0.2,-0.6) {$x$};
        \vertex (2) at (-2,0) [scale=.75] {};
        \node ($y$) [scale=0.9] at (-1.8,-0.6) {$y$};
        \vertex (3) at (-3,1.2) [scale=.75] {};
        \node ($y_1$) [scale=0.9] at (-2.3,1.4) {$y_1$};
        \vertex (4) at (-4.5,1.2) [scale=.75] {};
        \node ($y_2$) [scale=0.9] at (-5.2,1.4) {$y_2$};
        \vertex (5) at (-5.6,0.5) [scale=.75] {};
        \node ($y_3$) [scale=0.9] at (-6.4,0.5) {$y_3$};
        \vertex (6) at (-5.6,-0.5) [scale=.75] {};
        \node ($y_3'$) [scale=0.9] at (-6.4,-0.5) {$y_3'$};
        \vertex (7) at (-4.5,-1.2) [scale=.75] {};
        \node ($y_2'$) [scale=0.9] at (-5.1,-1.7) {$y_2'$};
        \vertex (8) at (-3,-1.2) [scale=.75] {};
        \node ($y_1'$) [scale=0.9] at (-2.5,-1.7) {$y_1'$};
        \vertex (9) at (2,0) [scale=.75] {};
        \node ($v$) [scale=0.9] at (2.4,-0.6) {$v$};
        \vertex (10) at (2,2) [scale=.75] {};
        \node ($x_1$) [scale=0.9] at (2.8,2) {$x_1$};
        \vertex (11) at (4,0) [scale=.75] {};
        \node ($x_3$) [scale=0.9] at (4.4,-0.6) {$x_3$};
        \vertex (12) at (2,-2) [scale=.75] {};
        \node ($x_2$) [scale=0.9] at (2.8,-2) {$x_2$};

        \path
        (1) edge (2)
        (2) edge (3)
        (3) edge (4)
        (4) edge (5)
        (5) edge (6)
        (6) edge (7)
        (7) edge (8)
        (8) edge (2)
        (1) edge (9)
        (9) edge (10)
        (9) edge (11)
        (9) edge (12)
        ;

        \draw (-4.4,1.32)..controls (-1.5,3.3) and (-0.2,0.1) ..(-0.1,0.12);

        \draw (-4.45,1.35)..controls (-1.5,4) and (1.8,2.2) ..(1.9,2.12);

        %the arrow 
        \vertex (200) at (6,0) [scale=.04,fill=black] {};
        \vertex (201) at (8,0) [scale=.01] {};
       \draw[thick][->,>=latex](200)--(201);

       \vertex (51) at (16,0) [scale=.75] {};
       \node ($x$) [scale=0.9] at (16.2,-0.6) {$x$};
       \vertex (52) at (14,0) [scale=.75] {};
       \node ($y$) [scale=0.9] at (14.2,-0.6) {$y$};
       \vertex (53) at (13,1.2) [scale=.75] {};
       \node ($y_1$) [scale=0.9] at (13.7,1.4) {$y_1$};
       \vertex (54) at (11.5,1.2) [scale=.75] {};
       \node ($y_2$) [scale=0.9] at (10.8,1.4) {$y_2$};
       \vertex (55) at (10.4,0.5) [scale=.75] {};
       \node ($y_3$) [scale=0.9] at (9.6,0.5) {$y_3$};
       \vertex (56) at (10.4,-0.5) [scale=.75] {};
       \node ($y_3'$) [scale=0.9] at (9.6,-0.5) {$y_3'$};
       \vertex (57) at (11.5,-1.2) [scale=.75] {};
       \node ($y_2'$) [scale=0.9] at (10.9,-1.7) {$y_2'$};
       \vertex (58) at (13,-1.2) [scale=.75] {};
       \node ($y_1'$) [scale=0.9] at (13.5,-1.7) {$y_1'$};
       \vertex (59) at (18,0) [scale=.75] {};
       \node ($v$) [scale=0.9] at (18.4,-0.6) {$v$};
       \vertex (60) at (18,2) [scale=.75] {};
       \node ($x_1$) [scale=0.9] at (18.8,2) {$x_1$};
       \vertex (61) at (20,0) [scale=.75] {};
       \node ($x_3$) [scale=0.9] at (20.4,-0.6) {$x_3$};
       \vertex (62) at (18,-2) [scale=.75] {};
       \node ($x_2$) [scale=0.9] at (18.8,-2) {$x_2$};

       \path
       (51) edge (52)
       (52) edge (53)
       (53) edge (54)
       (54) edge (55)
       (55) edge (56)
       (56) edge (57)
       (57) edge (58)
       (58) edge (52)
       (51) edge (59)
       (59) edge (60)
       (59) edge (61)
       (59) edge (62)
       ;

       \draw (14.1,0.1)..controls (15.7,2.5) and (17.8,2) ..(17.8,2.05);
       \draw (11.55,1.35)..controls (14.5,4) and (17.8,2.2) ..(17.9,2.12);
    \end{tikzpicture}
\end{center}
\vspace{1.5mm}
\par {\footnotesize \centerline{{\bf Fig. 9.} ~The subcases that $xy_2 \in E(G)$ and $x_1y \in E(G)$. \hypertarget{Fig9}}}
\end{figure}

\emph{Subcase 3.2.3.} $y_1x \in E(G)$ and $y_1x_1 \in E(G)$. Assume that $N(y_1') \cap \{x_2,x_3\} = \emptyset$. Clearly, $\{y_1, y_3'\}$ is a 1-isolating set of $G[X]$. Remark that $x_1y_1' \in E(G)$ does not work here. By Lemmas \ref{Lemma2.2} and \ref{Lemma2.3}, and by the induction hypothesis, $\iota_1(G) \leq |\{y_1, y_3'\}| + \iota_1(G-X) \leq 2+ \frac{1}{4}(n-8) = \frac{n}{4}$.

Assume that $N(y_1') \cap \{x_2,x_3\} \neq \emptyset$ and $y_1'x_2 \in E(G)$. Since $yy_1'x_2vxy$ is not an induced 5-cycle in $G$, we derive that $x_2y \in E(G)$, or $y_1'x \in E(G)$, or $x_2x \in E(G)$. If $x_2y \in E(G)$, then $y_1yx_2vx_1y_1$ is a 5-cycle in $G$, implying $x_1x_2 \in E(G)$. Take $D=\{y_1,y_2'\}$ if $y_3x_3 \notin E(G)$, and $D=\{y,y_3\}$ if $y_3x_3 \in E(G)$. Clearly, $D$ is a 1-isolating set of $G[X]$. By Lemmas \ref{Lemma2.2} and \ref{Lemma2.3}, and by the induction hypothesis, $\iota_1(G) \leq |D| + \iota_1(G-X) \leq 2+ \frac{1}{4}(n-8) = \frac{n}{4}$. If $y_1'x \in E(G)$, then $y_1yy_1'x_2vx_1y_1$ is a 6-cycle in $G$, implying $x_2y,x_2x_1 \in E(G)$. However, $y_1xy_1'x_2x_1y_1$ is an induced 5-cycle in $G$. Hence, $x_2y,y_1'x \notin E(G)$ and $x_2x \in E(G)$. Since $x_1y_1yy_1'x_2vx_1$ is a 6-cycle in $G$, $x_1y_1' \in E(G)$. However, $x_1y_2y_3y_3'y_2'y_1'x_1$ is an induced 6-cycle in $G$.

\emph{Subcase 3.2.4.} $xy_1 \in E(G)$ and $xx_1 \in E(G)$. If $y_1'x_1 \in E(G)$, then $y_2y_1yy_1'x_1y_2$ is an induced 5-cycle in $G$. So, let $y_1'x_1 \notin E(G)$. Assume that $N(y_1') \cap \{x_2,x_3\} \neq \emptyset$ and $y_1'x_2 \in E(G)$. Then, $y_1'x_2vxyy_1'$ is a 5-cycle in $G$. Since $G$ contains no induced 5-cycles and $d(x) = d(v) = \Delta = 4$, $x_2y \in E(G)$. Clearly, $y_2y_1yx_2vx_1y_2$ is a 6-cycle in $G$, and $x_2y_2 \in E(G)$. However, by $d(x_2)= \Delta=4$, $x_2y_2y_3y_3'y_2'y_1'x_2$ is an induced 6-cycle in $G$. Let $y_1'x_2 \notin E(G)$. By the symmetry of $x_2$ and $x_3$, $y_1'x_3 \notin E(G)$ and $N(y_1') \cap \{x_1,x_2,x_3\} = \emptyset$. Note that $\{y_1,y_3'\}$ is a 1-isolating set of $G[X]$. By Lemma \ref{Lemma2.2} and the induction hypothesis, $\iota_1(G) \leq |\{y_1,y_3'\}|+ \iota_1(G-X) \leq 2+\frac{1}{4}(n-8) = \frac{n}{4}$.

\emph{Subcase 3.2.5.} $x_1y_1 \in E(G)$ and $x_1x \in E(G)$.
We relabel the vertices of $G$ as follows: $y_2=y, y=y_2, y_3=y_1', y_3'=y_2', y_2'=y_3', y_1'=y_3, x=x_1$ and $x_1=x$. As an illustration in Fig. \hyperlink{Fig10}{10}, this subcase can come down to Subcase 3.2.4.

\begin{figure}[h!]
    \begin{center}
    \begin{tikzpicture}[scale=.48]
    \tikzstyle{vertex}=[circle, draw, inner sep=0pt, minimum size=6pt]
    \tikzset{vertexStyle/.append style={rectangle}}
        \vertex (1) at (-0,0) [scale=.75] {};
        \node ($x$) [scale=0.9] at (0.2,-0.6) {$x$};
        \vertex (2) at (-2,0) [scale=.75] {};
        \node ($y$) [scale=0.9] at (-1.8,-0.6) {$y$};
        \vertex (3) at (-3,1.2) [scale=.75] {};
        \node ($y_1$) [scale=0.9] at (-2.2,1.1) {$y_1$};
        \vertex (4) at (-4.5,1.2) [scale=.75] {};
        \node ($y_2$) [scale=0.9] at (-5.2,1.4) {$y_2$};
        \vertex (5) at (-5.6,0.5) [scale=.75] {};
        \node ($y_3$) [scale=0.9] at (-6.4,0.5) {$y_3$};
        \vertex (6) at (-5.6,-0.5) [scale=.75] {};
        \node ($y_3'$) [scale=0.9] at (-6.4,-0.5) {$y_3'$};
        \vertex (7) at (-4.5,-1.2) [scale=.75] {};
        \node ($y_2'$) [scale=0.9] at (-5.1,-1.7) {$y_2'$};
        \vertex (8) at (-3,-1.2) [scale=.75] {};
        \node ($y_1'$) [scale=0.9] at (-2.5,-1.7) {$y_1'$};
        \vertex (9) at (2,0) [scale=.75] {};
        \node ($v$) [scale=0.9] at (2.4,-0.6) {$v$};
        \vertex (10) at (2,2) [scale=.75] {};
        \node ($x_1$) [scale=0.9] at (2.8,2) {$x_1$};
        \vertex (11) at (4,0) [scale=.75] {};
        \node ($x_3$) [scale=0.9] at (4.4,-0.6) {$x_3$};
        \vertex (12) at (2,-2) [scale=.75] {};
        \node ($x_2$) [scale=0.9] at (2.8,-2) {$x_2$};

        \path
        (1) edge (2)
        (2) edge (3)
        (3) edge (4)
        (4) edge (5)
        (5) edge (6)
        (6) edge (7)
        (7) edge (8)
        (8) edge (2)
        (1) edge (9)
        (9) edge (10)
        (9) edge (11)
        (9) edge (12)
        ;

        \draw (-2.9,1.32)..controls (-0.4,3) and (1.8,2) ..(1.84,2.1);

        \draw (0.05,0.15)..controls (0.5,2) and (1.8,2.1) ..(1.82,2);

        \draw (-4.45,1.35)..controls (-1.5,4) and (1.8,2.2) ..(1.9,2.12);

        %the arrow 
        \vertex (200) at (6,0) [scale=.04,fill=black] {};
        \vertex (201) at (8,0) [scale=.01] {};
       \draw[thick][->,>=latex](200)--(201);

       \vertex (51) at (16,0) [scale=.75] {};
       \node ($x$) [scale=0.9] at (16.2,-0.6) {$x$};
       \vertex (52) at (14,0) [scale=.75] {};
       \node ($y$) [scale=0.9] at (14.2,-0.6) {$y$};
       \vertex (53) at (13,1.2) [scale=.75] {};
       \node ($y_1$) [scale=0.9] at (12.8,1.67) {$y_1$};
       \vertex (54) at (11.5,1.2) [scale=.75] {};
       \node ($y_2$) [scale=0.9] at (10.8,1.4) {$y_2$};
       \vertex (55) at (10.4,0.5) [scale=.75] {};
       \node ($y_3$) [scale=0.9] at (9.6,0.5) {$y_3$};
       \vertex (56) at (10.4,-0.5) [scale=.75] {};
       \node ($y_3'$) [scale=0.9] at (9.6,-0.5) {$y_3'$};
       \vertex (57) at (11.5,-1.2) [scale=.75] {};
       \node ($y_2'$) [scale=0.9] at (10.9,-1.7) {$y_2'$};
       \vertex (58) at (13,-1.2) [scale=.75] {};
       \node ($y_1'$) [scale=0.9] at (13.5,-1.7) {$y_1'$};
       \vertex (59) at (18,0) [scale=.75] {};
       \node ($v$) [scale=0.9] at (18.4,-0.6) {$v$};
       \vertex (60) at (18,2) [scale=.75] {};
       \node ($x_1$) [scale=0.9] at (18.8,2) {$x_1$};
       \vertex (61) at (20,0) [scale=.75] {};
       \node ($x_3$) [scale=0.9] at (20.4,-0.6) {$x_3$};
       \vertex (62) at (18,-2) [scale=.75] {};
       \node ($x_2$) [scale=0.9] at (18.8,-2) {$x_2$};
       
       \path
       (51) edge (52)
       (52) edge (53)
       (53) edge (54)
       (54) edge (55)
       (55) edge (56)
       (56) edge (57)
       (57) edge (58)
       (58) edge (52)
       (51) edge (59)
       (59) edge (60)
       (59) edge (61)
       (59) edge (62)
       ;

       \draw (13.1,1.32)..controls (15,1.8) and (16,0.16) ..(15.98,0.16);
       \draw (16.02,0.16)..controls (16.5,2) and (17.8,2.1) ..(17.82,2);
       \draw (11.55,1.35)..controls (14.5,4) and (17.8,2.2) ..(17.9,2.12);
    \end{tikzpicture}
\end{center}
\vspace{1.5mm}
\par {\footnotesize \centerline{{\bf Fig. 10.} ~The subcases that $x_1y_1, x_1x \in E(G)$ and $xy_1, xx_1 \in E(G)$. \hypertarget{Fig10}}}
\end{figure}

\emph{Subcase 3.3.} $H^* \cong C_{11}$. If $N(y_2) \cap \{x_1,x_2,x_3\} = \emptyset$ and $N(y_2') \cap \{x_1,x_2,x_3\} = \emptyset$, then $\{y,y_4,y_4'\}$ is a 1-isolating set of $G[X]$. By Lemmas \ref{Lemma2.2} and \ref{Lemma2.3}, and by the induction hypothesis, $\iota_1(G) \leq |\{y,y_4,y_4'\}| + \iota_1(G-X) \leq 3+\frac{1}{4}(n-12) = \frac{n}{4}$. So, let $N(y_2) \cap \{x_1,x_2,x_3\} \neq \emptyset$ or $N(y_2') \cap \{x_1,x_2,x_3\} \neq \emptyset$. Without loss of generality, we assume that $N(y_2) \cap \{x_1,x_2,x_3\} \neq \emptyset$ and $y_2x_1 \in E(G)$. It is easy to see that $y_2y_1yxvx_1y_2$ is a 6-cycle in $G$. Since $G$ does not contain induced 5- and 6-cycles, we derive that $xy_2 \in E(G)$, or $x_1y \in E(G)$, or $y_1x \in E(G)$ and $y_1x_1 \in E(G)$, or $xy_1 \in E(G)$ and $xx_1 \in E(G)$, or $x_1y_1 \in E(G)$ and $x_1x \in E(G)$.

\emph{Subcase 3.3.1.} $xy_2 \in E(G)$. Assume that $N(y) \cap \{x_2,x_3\} = \emptyset$ and $N(y_4) \cap \{x_2,x_3\} = \emptyset$. It is easy to see that $\{y_2,y_2',y_5'\}$ is a 1-isolating set of $G[X]$ now. By Lemmas \ref{Lemma2.2} and \ref{Lemma2.3}, and by the induction hypothesis, $\iota_1(G) \leq |\{y_2,y_2',y_5'\}|+ \iota_1(G-X) \leq 3+\frac{1}{4}(n-12) = \frac{n}{4}$.

Assume that $N(y) \cap \{x_2,x_3\} \neq \emptyset$ and $yx_2 \in E(G)$. Then, $yx_2vx_1y_2y_1y$ is a 6-cycle in $G$. Since $G$ does not contain induced 5- and 6-cycles, we derive that $y_1x_1 \in E(G)$ and $y_1x_2 \in E(G)$, or $x_2x_1 \in E(G)$ and $x_2y_1 \in E(G)$, or $x_1y_1 \in E(G)$ and $x_1x_2 \in E(G)$. If $y_1x_1 \in E(G)$ and $y_1x_2 \in E(G)$, then $xy_2y_1x_2vx$ is a 5-cycle in $G$, and then $xx_2 \in E(G)$. However, $y_1yxvx_1y_1$ is an induced 5-cycle in $G$. If $x_2x_1 \in E(G)$ and $x_2y_1 \in E(G)$, then $xy_2y_1x_2vx$ is a 5-cycle in $G$, and then $xy_1 \in E(G)$. However, $xyx_2x_1y_2x$ is an induced 5-cycle in $G$. So, we assume that $x_1y_1 \in E(G)$ and $x_1x_2 \in E(G)$. Since  $x_1x_2yxy_2x_1$ is a 5-cycle in $G$, $x_2x \in E(G)$. However, $y_1yxvx_1y_1$ is an induced 5-cycle in $G$, always obtaining a contradiction. This suggests that $N(y) \cap \{x_2,x_3\} = \emptyset$ in the following.

%Take 
%\begin{displaymath}
    %D = \left\{
    %\begin{array}{ll}
    %\{y,y_3,y_5'\},& \text{if}\ y_3x_3 \notin E(G),\\
    %\{y,y_3,y_3'\},& \text{if}\ %y_3x_3 \in E(G).\\
    %\end{array} \right.
%\end{displaymath}
%Assume that $y_2'x_3 \in E(G)$. Then $y_2'x_3vxyy_1'y_2'$ is an induced 6-cycle in $G$ if $x_3y_1' \notin E(G)$, and $yy_1'x_3vxy$ is an induced 5-cycle in $G$ if $x_3y_1' \in E(G)$, a contradiction. So $y_2'x_3 \notin E(G)$. Hence, $D$ is a 1-isolating set of $G[X]$. By the induction hypothesis, $\iota_1(G) \leq |D|+ \iota_1(G-X) \leq 3+\frac{1}{4}(n-12) = \frac{n}{4}$. 

Assume that $N(y_4) \cap \{x_2,x_3\} \neq \emptyset$ and $y_4x_2 \in E(G)$. Since $y_4y_3y_2xvx_2y_4$ can not be an induced 6-cycle in $G$, we derive that $xy_4 \in E(G)$, or $x_2y_3 \in E(G)$ and $x_2x \in E(G)$, or $y_3x_2\in E(G)$ and $y_3x \in E(G)$. For any subcase, $d(x)=\Delta=4$. If $x_2y_3,x_2x \in E(G)$, then $d(x_2)=\Delta=4$, and $y_1'x_2 \notin E(G)$. Suppose that $y_1'x_3 \in E(G)$. Since $N(y) \cap \{x_2,x_3\} = \emptyset$, $yy_1'x_3vxy$ is an induced 5-cycle in $G$, a contradiction. If $xy_4 \in E(G)$ or $y_3x_2,y_3x \in E(G)$, then since $N(y) \cap \{x_2,x_3\} = \emptyset$, $x_iy_1'\in E(G)$ for each $i \in \{2,3\}$; otherwise, $yy_1'x_ivxy$ is an induced 5-cycle in $G$. So, $N(y_1') \cap \{x_2,x_3\} = \emptyset$. Clearly, now $\{y_2,y_5,y_3'\}$ is a 1-isolating set of $G[X]$. It does not matter whether $y,y_1' \in N(x_1)$ or not. By Lemmas \ref{Lemma2.2} and \ref{Lemma2.3}, and by the induction hypothesis, $\iota_1(G) \leq |\{y_2,y_5,y_3'\}|+ \iota_1(G-X) \leq 3+\frac{1}{4}(n-12) = \frac{n}{4}$.

%In each of these three subcases, we have $d(x)=4$. Note that $N(y) \cap \{x_2,x_3\} = \emptyset$. Then $yy_1'x_2vxy$ is an induced 5-cycle in $G$ if $y_1'x_2 \in E(G)$, and $yy_1'x_3vxy$ is an induced 5-cycle in $G$ if $y_1'x_3 \in E(G)$, a contradiction. So, $y_1'x_2 \notin E(G)$ and $y_1'x_3 \notin E(G)$. It follows that $\{y_2,y_5,y_3'\}$ is a 1-isolating set of $G[X]$. By the induction hypothesis, $\iota_1(G) \leq |\{y_2,y_5,y_3'\}|+ \iota_1(G-X) \leq 3+\frac{1}{4}(n-12) = \frac{n}{4}$.

%If $x_2y_3 \in E(G)$ and $x_2x \in E(G)$, then $d(x_2)=4$. Since $N(y) \cap \{x_2,x_3\} = \emptyset$, $y_1'x_3 \notin E(G)$. Otherwise, $yy_1'x_3vxy$ is an induced 5-cycle in $G$, a contradiction. Hence, $D=\{y_2,y_5,y_3'\}$ is a 1-isolating set of $G[X]$. If $y_3x_2\in E(G)$ and $y_3x \in E(G)$, then $y_2y_3x_2vx_1y_2$ is a 5-cycle in $G$, implying that $x_1x_2 \in E(G)$. Since $N(y) \cap \{x_2,x_3\} = \emptyset$, $y_1'x_3 \notin E(G)$. Hence, $D=\{y_2,y_5,y_3'\}$ is a 1-isolating set of $G[X]$. If $xy_4 \in E(G)$, then $d(x)=\Delta=4$. Since $N(y) \cap \{x_2,x_3\} = \emptyset$, $y_1'x_3 \notin E(G)$. Note that if $y_1'x_2 \in E(G)$, then $yy_1'x_2vxy$ is a 5-cycle in $G$, and then $x_2y \in E(G)$. However, $y_4y_3y_2x_1vx_2y_4$ is an induced 6-cycle in $G$ if $y_3x_1 \notin E(G)$, and $y_4y_3x_1vx_2y_4$ is an induced 5-cycle in $G$ if $y_3x_1 \in E(G)$, a contradiction. Hence, $N(y_1') \cap \{x_2,x_3\} = \emptyset$, and $D=\{y_2,y_5,y_3'\}$ is a 1-isolating set of $G[X]$. By the induction hypothesis, $\iota_1(G) \leq |D|+ \iota_1(G-X) \leq 3+\frac{1}{4}(n-12) = \frac{n}{4}$.

\emph{Subcase 3.3.2.} $x_1y \in E(G)$. We relabel the vertices of $G$ as follows: $y_2=y, y=y_2, y_3=y_1', y_4=y_2', y_5=y_3', y_5'=y_4', y_4'=y_5', y_3'=y_5, y_2'=y_4, y_1'=y_3, x=x_1$ and $x_1=x$. Thus, this subcase can be come down to Subcase 3.3.1.

\emph{Subcase 3.3.3.} $y_1x \in E(G)$ and $y_1x_1 \in E(G)$. We further consider the edges between $\{y_1',y_2'\}$ and $\{x_2,x_3\}$.

Assume that $N(y_1') \cap \{x_2,x_3\} = \emptyset$ and $N(y_2') \cap \{x_2,x_3\} = \emptyset$. Clearly, $\{y_1,y_4,y_4'\}$ is a 1-isolating set of $G[X]$. By Lemmas \ref{Lemma2.2} and \ref{Lemma2.3}, and by the induction hypothesis, $\iota_1(G) \leq |\{y_1,y_4,y_4'\}|+ \iota_1(G-X) \leq 3+\frac{1}{4}(n-12)=\frac{n}{4}$. 

Assume that $N(y_1') \cap \{x_2,x_3\} \neq \emptyset$  and $y_1'x_2 \in E(G)$. Since $yy_1'x_2vxy$ is a 5-cycle in $G$, we derive that $xy_1' \in E(G)$, or $xx_2 \in E(G)$, or $x_2y \in E(G)$. If $xy_1' \in E(G)$, then $y_1x_1vx_2y_1'yy_1$ is a 6-cycle in $G$. Since $d(y_1)=d(y_1')=d(v)=\Delta=4$, $x_1x_2 \in E(G)$ and $yx_2 \in E(G)$. However, $y_1x_1x_2y_1'xy_1$ is an induced 5-cycle in $G$. If $xx_2 \in E(G)$, then $y_1yy_1'x_2vx_1$ is a 6-cycle in $G$, implying that $x_1y_1' \in E(G)$. However, $yy_1'x_1vxy$ is an induced 5-cycle in $G$. If $x_2y \in E(G)$, then $y_1yx_2vx_1y_1$ is a 5-cycle in $G$. Since $d(y_1)=d(y)=d(v)=\Delta=4$, $x_1x_2 \in E(G)$. Take 
\begin{displaymath}
    D = \left\{
    \begin{array}{ll}
    \{y_1,y_2',y_5'\},& \text{if}\ y_3x_3 \notin E(G) \ \text{and}\ y_4x_3 \notin E(G),\\
    \{y,y_3,y_3'\},& \text{if}\ y_3x_3 \in E(G),\\
    \{y_1,y_2',y_4\},& \text{if}\ y_4x_3 \in E(G).\\
    \end{array} \right.
\end{displaymath}
Clearly, $D$ is a 1-isolating set of $G[X]$. By Lemmas \ref{Lemma2.2} and \ref{Lemma2.3}, and by the induction hypothesis, $\iota_1(G) \leq |D|+ \iota_1(G-X) \leq 3+\frac{1}{4}(n-12) = \frac{n}{4}$. 

Assume that $N(y_2') \cap \{x_2,x_3\} \neq \emptyset$ and $y_2'x_2 \in E(G)$. If $y_3x_3 \in E(G)$, then $y_3y_2y_1xvx_3y_3$ is a 6-cycle in $G$. By the non-existence of induced 5- and 6-cycles in $G$, we derive that $y_3x \in E(G)$, or $x_3y_2 \in E(G)$ and $x_3x \in E(G)$. If $y_3x \in E(G)$, then $y_2y_3xvx_1y_2$ is an induced 5-cycle in $G$, a contradiction. If $x_3y_2 \in E(G)$ and $x_3x \in E(G)$, then $d(x_3)=\Delta=4$. It is easy to see that, $D=\{y_1,y_4,y_2'\}$ is a 1-isolating set of $G[X]$, and $x_1,x_2 \in N(D)$. By Lemma \ref{Lemma2.2} and the induction hypothesis, $\iota_1(G) \leq |\{y_1,y_4,y_2'\}|+ \iota_1(G-X) \leq 3+\frac{1}{4}(n-12) = \frac{n}{4}$.

If $y_3x_3 \notin E(G)$, then we take
\begin{displaymath}
    D = \left\{
    \begin{array}{ll}
    \{y_1,y_2',y_5'\},& \text{if}\ y_4x_3 \notin E(G),\\
    \{y_1,y_2',y_4\},& \text{if}\ y_4x_3 \in E(G).\\
    \end{array} \right.
\end{displaymath}
Clearly, $D$ is a 1-isolating set of $G[X]$. By Lemmas \ref{Lemma2.2} and \ref{Lemma2.3}, and by the induction hypothesis, $\iota_1(G) \leq |D|+ \iota_1(G-X) \leq 3+\frac{1}{4}(n-12) = \frac{n}{4}$.

\emph{Subcase 3.3.4.} $xy_1 \in E(G)$ and $xx_1 \in E(G)$. Clearly, $y_2y_1yy_1'x_1y_2$ is an induced 5-cycle in $G$ if $y_1'x_1 \in E(G)$, and $yy_1'y_2'x_1xy$ is an induced 5-cycle in $G$ if $y_2'x_1 \in E(G)$. So, it follows that $y_1'x_1,y_2'x_1 \notin E(G)$. We further consider the edges between $\{y_1',y_2'\}$ and $\{x_2,x_3\}$.

Assume that $N(y_1') \cap \{x_2,x_3\} = \emptyset$ and $N(y_2') \cap \{x_2,x_3\} = \emptyset$. Clearly, $\{y_1,y_4,y_4'\}$ is a 1-isolating set of $G[X]$. By Lemma \ref{Lemma2.2} and the induction hypothesis, $\iota_1(G) \leq |\{y_1,y_4,y_4'\}|+ \iota_1(G-X) \leq 3+\frac{1}{4}(n-12)=\frac{n}{4}$.

Assume that $N(y_1') \cap \{x_2,x_3\} \neq \emptyset$ and $y_1'x_2 \in E(G)$. Since $yy_1'x_2vxy$ is a 5-cycle in $G$, and $d(x)=d(v)=\Delta=4$, $x_2y \in E(G)$. Furthermore, $y_2y_1yx_2vx_1y_2$ is a 6-cycle in $G$. By the non-existence of induced 5- and 6-cycles in $G$, $x_2y_2 \in E(G)$. However, $y_2y_1xvx_2y_2$ is an induced 5-cycle in $G$, a contradiction.

Assume that $N(y_2') \cap \{x_2,x_3\} \neq \emptyset$ and $y_2'x_2 \in E(G)$. Since $yy_1'y_2'x_2vxy$ is a 6-cycle in $G$, $x_2y \in E(G)$. Furthermore, $y_2y_1yx_2vx_1y_2$ is a 6-cycle in $G$, and thus, $x_2y_2 \in E(G)$. However, $y_2y_1xvx_2y_2$ is an induced 5-cycle in $G$, a contradiction.

\emph{Subcase 3.3.5.} $x_1y_1 \in E(G)$ and $x_1x \in E(G)$. We relabel the vertices of $G$ as follows: $y_2=y, y=y_2, y_3=y_1', y_4=y_2', y_5=y_3', y_5'=y_4', y_4'=y_5', y_3'=y_5, y_2'=y_4, y_1'=y_3, x=x_1$ and $x_1=x$. Thus, this subcase comes down to Subcase 3.3.4.

\vspace{3mm}
This completes the proof of Theorem \ref{Theorem1.7}.
\end{proof}

\section{\large Conclusions}

In this paper, we study the 1-isolation number of graphs without short cycles, and establish two sharp upper bounds on $\iota_1(G)$. More specifically, we prove that if $G \notin \{P_3,C_3,C_7,C_{11}\}$ is a connected graph of order $n$ without $6$-cycles, or without induced 5- and 6-cycles, then $\iota_1(G) \leq \frac{n}{4}$. This in fact extends a result of Caro and Hansberg \cite{Caro2017} that if $T$ is a tree of order $n$ and different from $P_3$, then $\iota_1(T) \leq \frac{n}{4}$, and a result of Zhang and Wu \cite{Zhang2021} that if $G \notin \{P_3,C_7,C_{11}\}$ is a connected graph of order $n$ with girth at least 7, then $\iota_1(G) \leq \frac{n}{4}$.

\vspace{3mm}
A more interesting and profound problem is proposed as follows.

\begin{problem}\label{Problem5.1}
Let $G$ be a connected graph of order $n$ without induced 6-cycles. Determine the exact value of $\lim \sup_{n \to \infty} \frac{\iota_1(G)}{n}$. Is it $\frac{1}{4}$?
\end{problem}

%We believe that its answer is affirmative.

\vspace{3mm}
\noindent{\large\bf Acknowledgments}
\vspace{3mm}

This work is supported by National Natural Science Foundation of China (No. 12171402).

%The authors are grateful to the anonymous reviewer(s) for their careful reading and helpful comments.

%\newpage

\end{document}